\documentclass[a4paper,pdftex,reqno,10pt]{amsart}

\allowdisplaybreaks[1]

\usepackage[german,english]{babel}
\usepackage{enumerate}
\usepackage{amsmath,amssymb}
\usepackage{geometry}
\usepackage{float}
\usepackage{rotating}

\newcommand{\gauss}[3]{\genfrac{[}{]}{0pt}{}{#1}{#2}_{#3}}

\newcommand{\G}[3]{\mathcal{G}_{#3}(#1,#2)}

\newcommand{\I}[2]{\ensuremath{\mathcal{I}\left(#1,#2\right)}}
\DeclareMathOperator{\PG}{PG}
\DeclareMathOperator{\Aut}{Aut} 
\DeclareMathOperator{\GL}{GL}
\DeclareMathOperator{\PGL}{PGL}
\DeclareMathOperator{\PGaL}{P\Gamma L} 
\DeclareMathOperator{\st}{st}
\newcommand{\denum}{\delta}

\newcommand{\sdist}{\mathrm{d}_{\mathrm{S}}}

\newcommand{\smax}{\mathrm{A}}

\newcommand{\F}{\mathbb{F}} 

\newcommand{\gabidulin}{\mathcal{G}}

\newtheorem{theorem}{Theorem}
\newtheorem{corollary}{Corollary} 
\newtheorem{lemma}[theorem]{Lemma}
\newtheorem{proposition}{Proposition}

\newtheorem*{problem}{Problem}

\theoremstyle{definition} 
\newtheorem{definition}[theorem]{Definition}

\title[Binary Subspace Codes in Small Ambient Spaces]{Binary Subspace Codes in Small Ambient Spaces}

\author[Daniel Heinlein and Sascha Kurz]{}

\email{daniel.heinlein@uni-bayreuth.de}
\email{sascha.kurz@uni-bayreuth.de}

\thanks{The authors were supported by the DFG project ``Ganz\-zahlige Optimierungs\-modelle f\"ur
  Subspace Codes und endliche Geometrie'' (DFG grants KU~2430/3-1, WA~1666/9-1).}

\begin{document}
\maketitle

\centerline{\scshape Daniel Heinlein} \medskip {\footnotesize
  \centerline{Mathematisches Institut, Universit\"at Bayreuth, D-95440
    Bayreuth, Germany} }

\medskip

\centerline{\scshape Sascha Kurz} \medskip {\footnotesize
  \centerline{Mathematisches Institut, Universit\"at Bayreuth, D-95440
    Bayreuth, Germany} }

\bigskip

\begin{abstract}
  Codes in finite projective spaces equipped with the subspace distance have been proposed for error control 
  in random linear network coding. Here we collect the present knowledge on lower and upper bounds for binary 
  subspace codes for projective dimensions of at most $7$. We obtain several improvements of the bounds and 
  perform two classifications of optimal subspace codes, which are unknown so far in the literature.
\end{abstract}


\section{Introduction}
\label{sec_introduction}

For a prime power $q>1$ let $\mathbb{F}_q$ be the finite field with
$q$ elements and $\mathbb{F}_q^v$ the standard vector
space of dimension $v\geq 0$ over $\mathbb{F}_q$. The set of all subspaces
of $\mathbb{F}_q^v$, denoted by $\PG(v-1,\F_q)$, is a finite modular geometric
lattice with meet $X\wedge Y=X\cap Y$ and join $X\vee Y=X+Y$. 
Network-error-correcting codes developed by Cai and Yeung, see \cite{cai2006network,yeung2006network},  
use subspaces of $\PG(v-1,\F_q)$ as codewords. 
We call any set $\mathcal{C}$ of subspaces of $\F_q^v$ a \emph{$q$-ary subspace code}. 
An information-theoretic analysis of the Koetter-Kschischang-Silva model, see 
\cite{koetter2008coding,silva2008rank}, motivates the use of 
the so-called \emph{subspace distance}
\begin{equation}
  \begin{aligned}
    \sdist(X,Y)&=\dim(X+Y)-\dim(X\cap Y)\\
    &=\dim(X)+\dim(Y)-2\cdot\dim(X\cap Y)\\
    &=2\cdot\dim(X+Y)-\dim(X)-\dim(Y).
  \end{aligned}  
\end{equation} 
With this the \emph{minimum distance} in the subspace metric of a
subspace code $\mathcal{C}$ containing at least two codewords is defined as
\begin{equation}
  \sdist(\mathcal{C}):=\min\left\{\sdist(X,Y)\,:\,X,Y\in\mathcal{C},
    X\neq Y\right\}.
\end{equation}
If $\#\mathcal{C}\leq 1$ we formally set $\sdist(\mathcal{C})=\infty$. 
A subspace code $\mathcal{C}$ is called a \emph{constant-dimension code} if all
codewords have the same dimension. A $k$-dimensional subspace is also called $k$-subspace 
and we also write point, line, plane, and solid for $1$-, $2$-, $3$-, and $4$-subspaces, 
respectively. Hyperplanes are $(v-1)$-subspaces in $\F_q^v$, i.e., they have co-dimension 
one. The set of $k$-subspaces is denoted by $\G{v}{k}{q}$.

\begin{definition}
  A \emph{$q$-ary $(v,M,d)$ subspace code}, also referred to as a
  \emph{subspace code with parameters $(v,M,d)_q$}, is a set
  $\mathcal{C}$ of subspaces of $\F_q^v$ with $M=\#\mathcal{C}$ and minimum subspace distance $d$. 
  The \emph{dimension distribution} of $\mathcal{C}$ is the sequence
  $\denum(\mathcal{C})=\left(\delta_0,\delta_1,\dots,\delta_v\right)$
  defined by $\delta_k=\delta_k(\mathcal{C})=\#\{X\in\mathcal{C}\,:\,\dim(X)=k\}$. Two subspace
  codes $\mathcal{C}_1,\mathcal{C}_2$ are said to be \emph{isomorphic}
  if there exists an isometry (with respect to the subspace metric)
  $\phi\colon \F_q^v\to \F_q^v$ between their ambient spaces satisfying
  $\phi(\mathcal{C}_1)=\mathcal{C}_2$.
\end{definition}
The dimension distribution of a subspace 
code may be seen as a $q$-analogue of the Hamming weight distribution
of an ordinary block code. As in the block code case, the quantities
$\delta_k=\delta_k(\mathcal{C})$ are non-negative integers satisfying
$\sum_{k=0}^v \delta_k=M=\#\mathcal{C}$.
\begin{problem}  For a given field size $q$, dimension $v\geq 1$ of the ambient space 
  and minimum distance $d\in\{1,\dots,v\}$ determine the maximum size
  $\smax_q(v,d)=M$ of a $q$-ary $(v,M,d)$ subspace code   and---as a refinement---classify 
  the corresponding optimal codes up to subspace code isomorphism.
\end{problem}
More generally, for subsets $T\subseteq\{0,1,\dots,v\}$, we denote the maximum size
of a $(v,M,d')_q$ subspace code $\mathcal{C}$ with $d'\geq d$ and
$\delta_k(\mathcal{C})=0$ for all $k\in\{0,1,\dots,v\}\setminus T$
by $\smax_q(v,d;T)$, and refer to subspace codes subject to this
dimension restriction accordingly as $(v,M,d;T)_q$ codes.
For $T=\{0,\dots,v\}$ we obtain (general) subspace codes and for $T=\{k\}$ we obtain constant-dimension codes
and are able to treat both cases, as well as generalizations, in a common notation. 
Note that $\smax_q(v,d;T)\geq\smax_q(v,d';T)$ for all $1\leq d\leq d'\leq v$, 
$\smax_q(v,d;T)\leq\smax_q(v,d;T')$ for all $T\subseteq T'\subseteq\{0,1,\dots,v\}$, and 
$\smax_q(v,d;T\cup T')\le \smax_q(v,d;T)+\smax_q(v,d;T')$ for all $T, T'\subseteq\{0,1,\dots,v\}$, see 
e.g.\ \cite[Lemma 2.3]{honold2016constructions}.

While a lot of research has been done on the determination of the numbers 
$\smax_q(v,d;k)=\smax_q(v,d;\{k\})$, i.e., the constant-dimension case, see e.g.\ 
\cite{etzion2016galois,greferath2018network,heinlein2016tables}, only very 
few results are known for $\#T>1$, see e.g.\ \cite{honold2016constructions}. The purpose of 
this paper is to advance the knowledge in the mixed-dimension case and partially solve the 
above problem for $q=2$ and small $v$. 

The Gaussian binomial coefficient $\gauss{v}{k}{q}=\prod_{i=0}^{k-1}\frac{q^{v-i}-1}{q^{k-i}-1}$
gives the number of $k$-dimensional subspaces of $\F_q^v$. Since these numbers grow very quickly, 
especially for $k\approx v/2$, the exact determination of $\smax_q(v,d)$ appears to be an intricate
task---except for some special cases. Even more challenging is the
refined problem of enumerating the isomorphism types of the
corresponding optimal subspace codes (i.e., those of size
$\smax_q(v,d)$). 

In this paper we found new lower bounds for $\smax_2(6,3)$, $\smax_2(7,3)$, and $\smax_2(8,3)$, improved 
upper bounds for $\smax_2(6,3)$, $\smax_2(8,4)$, and $\smax_2(8,5)$, and classify the optimal codes for 
$\smax_2(8,6)$ and $\smax_2(8,7)$.

The remaining part of this paper is structured as follows. In Section~\ref{sec_preliminaries} we introduce 
some notation and preliminary facts. Dimensions $v\le 5$ are treated in Section~\ref{sec_dim_small}. For the 
cases $v\in\{6,7,8\}$ we have devoted sections \ref{sec_dim_6}, \ref{sec_dim_7}, and \ref{sec_dim_8}, respectively. 
A brief summary and conclusion is drawn in Section~\ref{sec_conclusion}.

\section{Preliminaries}
\label{sec_preliminaries}

In this section we summarize some notation and well-known insights that will
be used in the later parts of the paper.

\subsection{Gaussian elimination and representations of subspaces}
Elements in $\PG(v-1,\F_q)$ can be represented by matrices. To this end, let $A\in\mathbb{F}_q^{k\times v}$ be a matrix 
of full rank, i.e., rank $k$. The row-space of $A$ forms a $k$-dimensional subspace of $\F_q^v$, so that we call 
$A$ a \emph{generator matrix} of an element of $\G{v}{k}{q}$. Since the  application of the Gaussian elimination
 algorithm onto a generator matrix $A$ does not change the row-space, we can restrict ourselves onto generator matrices which are 
in \emph{reduced row echelon form} (rref), i.e., the matrix has the shape resulting from a Gaussian elimination.  
This gives a unique and well-known representation. We denote the underlying bijection by $\tau:\G{v}{k}{q}\rightarrow
\left\{A'\in \mathbb{F}_q^{k\times v}\,:\,\operatorname{rk}(A')=k,\, A'\text{ in rref}\right\}$. Slightly 
abusing notation, we apply $\tau$ for all values $1\le k\le v$ without referring to the involved parameters, which 
will be clear from the context. Given a matrix $A\in \mathbb{F}_q^{k\times v}$ of full rank we denote by $p(A)\in\mathbb{F}_2^v$ 
the binary vector whose $k$ $1$-entries coincide with the pivot columns of $A$. For each subspace $U$ of $\F_q^v$ we use the 
abbreviation $p(U)=p(\tau(U))$. Note that the number of ones in $p(U)$, i.e., the weight $\operatorname{wt}(p(U))$, coincides 
with the (vector-space) dimension $\dim(U)=k$. 

\subsection{Maximum rank distance codes}
For the set $\mathbb{F}_q^{m\times n}$ of $m\times n$-matrices over $\mathbb{F}_q$ the 
mapping $d\colon \mathbb{F}_q^{m\times n}\times \mathbb{F}_q^{m\times n}\to\mathbb{N}$ defined 
by $d(X,Y)=\operatorname{rank}(X-Y)$ is a metric, the \emph{rank distance} on $\mathbb{F}_q^{m\times n}$. 
Any subset $\mathcal{C}$ of $\mathbb{F}_q^{m\times n}$ is called \emph{rank-metric code} with 
\emph{minimum distance} $d(\mathcal{C})=\min\{d(X,Y)\,:\,X,Y\in\mathcal{C}, X\neq Y\}$ for $\#\mathcal{C}\ge 2$ 
and $d(\mathcal{C})=\infty$ otherwise. The Singleton-like bound
$$
  \#\mathcal{C}\le q^{\max\{m,n\}\cdot \left(\min\{m,n\}-d(\mathcal{C})+1\right)} 
$$ 
was proven in \cite{delsarte1978bilinear}. If the bound is met, $\mathcal{C}$ is called \emph{maximum rank distance} 
(MRD) code. They exist for all parameters and can e.g.\ be constructed using \emph{linearized polynomials}, i.e., 
$f(x)=\sum_{i=0}^{l} \alpha_ix^{q^i}\in\mathbb{F}_{q^n}[x]$, which are in bijection to $n\times n$-matrices over 
$\F_q$. If the domain for $x$ is restricted to a fixed $k$-dimensional $\F_q$-subspace $W$ of $\F_{q^n}$, then 
we obtain a bijection to $k\times n$-matrices over $\F_q$. The parameter $l$ and the $\alpha_i$ determine the rank. 
\begin{theorem}(Cf.\ \cite{delsarte1978bilinear})
  For integers $1\le\delta \le k\le n$ and a $k$-dimensional $\F_q$-subspace $W$ of $\F_{q^n}$ the union of
  \begin{equation*}
    G'(a_0,\dots,a_{k-\delta})
    =\bigl\{a_0x+a_1x^q+a_2x^{q^2}+\dots+a_{k-\delta}x^{q^{k-\delta}}\,:\,x\in W\bigr\}
  \end{equation*}  
  with $a_i\in\F_{q^n}$ gives an $k\times n$ MRD code with minimum rank distance $\delta$.
\end{theorem}
Those codes are often called \emph{Gabidulin codes}, referring to \cite{gabidulin1985theory}. Let $I_k$ denote $k\times k$ 
unit matrix (for an arbitrary field size $q$) and $A|B$ denote the column-wise concatenation of two matrices with the same 
number of rows. For each matrix $M\in\F_q^{k\times n}$ we can obtain a $k$-subspace in $\F_q^{n+k}$ via 
$\tau^{-1}(I_k|M)$, which is called \emph{lifting}. Applying the lifting construction to MRD codes gives so-called 
\emph{lifted MRD} (LMRD) codes. Specialized to Gabidulin codes, a $q$-ary lifted Gabidulin code $\gabidulin=\gabidulin_{v,k,\delta}$
has parameters $(v,q^{(k-\delta+1)(v-k)},2\delta;k)$, where $1\leq\delta\leq k\leq v/2$. It can be defined in a coordinate-free
manner as follows, see\ e.g.\ \cite[Sect.~2.5]{honold2015optimal}: The
ambient space is taken as $V=W\times\F_{q^n}$, where $n=v-k$ and $W$
denotes a fixed $k$-dimensional $\F_q$-subspace of $\F_{q^n}$, and
$\gabidulin$ consists of all subspaces
\begin{equation*}
  G(a_0,\dots,a_{k-\delta})
  =\bigl\{(x,a_0x+a_1x^q+a_2x^{q^2}+\dots+a_{k-\delta}x^{q^{k-\delta}})\,:\,x\in W\bigr\}
\end{equation*}
with $a_i\in\F_{q^n}$. The code $\gabidulin$ forms a geometrically quite regular object. The
most significant property, shared by all lifted MRD codes with the
same parameters, is that $\gabidulin$ forms an exact $1$-cover of the
set of all ($k-\delta$)-subspaces of $V$ that are disjoint\footnote{We say that two subspaces $U$ and 
$W$ are disjoint or intersect trivially if $\dim(U\cap W)=0$.} from the
special $(v-k)$-subspace $S=\{0\}\times\F_{q^n}$.

\subsection{The automorphism group of $\bigl(\PG(v-1,\F_q),\sdist\bigr)$}
Let us start with a description of the automorphism group of the metric space $\PG(v-1,\F_q)$ relative to the subspace
distance. The linear group $\GL(v,\F_q)$ acts on $\PG(v-1,\F_q)$ as a group of $\F_q$-linear isometries. Whenever $q$ is not
prime there are additional semilinear isometries arising from $\Aut(\F_q)$. Moreover, mapping a subspace
$X\subseteq\F_q^v$ to its dual code $X^\perp$ (with respect to the standard inner product) respects the subspace distance and 
hence yields a further automorphism $\pi$ of the metric space $\PG(v-1,\F_q)$. However, $\pi$ reverses the dimension distribution 
of a subspace code. This implies $\smax_q(v,d;T)=\smax_q(v,d;v-T)$, where $v-T=\{v-t\,:\,t\in T\}$. 

\begin{theorem}(E.g.~\cite[Theorem 2.1]{honold2016constructions})\\
  \label{thm:autos}
  Suppose that $v\geq 3$. The automorphism group $G$ of
  $\PG(v-1,\F_q)$, viewed as a metric space with respect to the
  subspace distance, is generated by $\GL(v,\F_q)$,
  $\Aut(\F_q)$ and $\pi$. More precisely, $G$ is the
  semidirect product of the projective general semilinear group
  $\PGaL(v,\F_q)$ with a group of order $2$ acting by matrix
  transposition on $\PGL(v,\F_q)$ and trivially on $\Aut(\F_q)$.  
\end{theorem}
In our case $q=2$ the semilinear part is void and we mostly will not use $\pi$, 
so that the isomorphism problem for subspace codes reduces to the determination 
of the orbits of $\GL(v,\F_2)$. 
Note that $\GL(v,\F_q)$ acts transitively on $\G{v}{k}{q}$. Moreover, 
for any integer triple $a,b,c$ satisfying $0\leq a,b\leq v$ and $\max\{0,a+b-v\}\leq c\leq\min\{a,b\}$ the group $\GL(v,\F_q)$ acts
transitively on ordered pairs of subspaces $(X,Y)$ of $\F_q^v$ with $\dim(X)=a$, $\dim(Y)=b$, and $\dim(X\cap
Y)=c$. In other words, the corresponding Grassmann graph is distance-regular, where the graph distance is half 
the subspace distance. In the later parts of the paper we repeatedly classify subspace codes up to isomorphism. 
To this end we utilize a software package developed by Thomas Feulner, see \cite{feulner2009automorphism,ubt_epub42}. 

\subsection{Optimization problems for $\smax_q(v,d;T)$}
\label{subsec_opt_formulations}
The problem of the determination of $\smax_q(v,d;T)$ can be easily described as a maximum clique problem. To this end we 
build up a graph $\mathcal{G}$ with the subspaces of $\F_q^v$ with dimensions in $T$ as vertices. Two vertices are joined by 
an edge iff the subspace distance between the two corresponding subspaces is at least $d$. This may be directly translated to 
an integer linear programming (ILP) formulation. However, tighter, with respect to the integrality gap, i.e., the difference 
between the optimal target value of the linear relaxation and the original problem, formulations are possible. For modeling 
variants in the constant-dimension case we refer e.g.\ to \cite{kohnert2008construction}. In order to allow a compact representation, we
denote the ball in $\F_q^v$ with radius $r$, with respect to the subspace distance, and center $W$, for all $W\in\PG(v-1,\F_q)$, 
by $B_r(W)=\left\{U\le\F_q^v\,:\, \sdist(u,W)\le r\right\}$. 
\begin{proposition}
  \label{prop_ILP1}
  For odd $d$ and all suitable other parameters we have:
  \begin{align*}
    \smax_q(v,d;T) &= \max \sum_{k\in T} \delta_k&&\text{subject to}&\\
    \sum_{U\in B_{(d-1)/2}(W)} x_U &\le  1&&\forall W\le \F_q^v&\\
    \sum_{U\in\G{v}{k}{q}} x_U &= \delta_k&&\forall 0\le k\le v&\\
    x_U&\in \{0,1\}&&\forall U\le\F_q^v.&
  \end{align*}
\end{proposition} 
\begin{proof}
  Whenever $\dim(U)\in T$, the binary variables $x_U$ have the interpretation $x_U=1$ iff $U\in\mathcal{C}$. The target function counts 
  the number of codewords with dimensions in $T$. (If $\dim(U)\notin T$, then we may set $x_U=0$ without violating any constraints or 
  changing the target function.) It remains to check that each feasible solution of the ILP above  corresponds to a subspace code with minimum subspace distance at 
  least $d$ and that each such code corresponds to a feasible solution of the ILP. For the latter, consider $U,U'\in\mathcal{C}$. Since 
  $\sdist(U,U')\ge d$, not both $U$ and $U'$ can be contained in $B_{(d-1)/2}(W)$ as $d$ is a metric. For the other direction we need to 
  ensure that for all subspaces $U$, $U'$ in $\F_q^v$ with $\sdist(U,U')\le d-1$ there exists $W\le \F_q^v$ with $U,U'\in B_{(d-1)/2}(W)$. 
  To this end note that $\sdist(X,Y)=\dim(X)+\dim(Y)-2\dim(X\cap Y)$ implies $d_1:=\sdist(U,W')=\dim(U)-\dim(W')$ and 
  $d_2:=\sdist(U',W')=\dim(U')-\dim(W')$ for $W'=U\cap U'$, so that $\sdist(U,U')=\sdist(U,W')+\sdist(U',W')$, i.e., $d_1+d_2\le d-1$. 
  W.l.o.g.\ we assume $\dim(U)\le \dim(U')$, so that $d_1\le d_2$. Write $U'=W'\oplus Z$ and choose $W''\le Z$ of dimension 
  $d_3=\left\lfloor\left(d_2-d_1\right)/2\right\rfloor\le\dim(Z)$. Then, for $W=\langle W',W''\rangle$ we have $\sdist(U,W)=d_1+d_3$, 
  $\sdist(U',W)=d_2-d_3$, $d_1+d_3\le (d-1)/2$, and $d_2-d_3\le (d-1)/2$.     
\end{proof}

If $T\neq\{0,1,\dots,v\}$, then some variables and constraints might be redundant, which however 
is automatically observed in the presolving step of the state-of-the-art ILP solvers. Based on known upper bounds on $\smax_q(v',d';k)$ 
the additional inequalities 
\begin{equation}
  \label{ie_add}
  \sum_{U\in\G{v}{k}{q}\,:\, U\text{ incident with } L} x_U +\mathbf{1}_{|k-l|<d}\cdot ax_L\le a
\end{equation} 
for all $L\in\G{v}{l}{q}$, where $0\le k\le v$ and $0\le l\le v$ are integers, $a\ge \smax_q(l,d;k)$ for $k\le l$, $a\ge \smax_q(v-l,d,k-l)$ for 
$k>l$, and the indicator function $\mathbf{1}_{|k-l|<d}$ is $1$ if $|k-l|<d$ and $0$ otherwise, are valid. For $k\le l$ the condition 
$U$ incident with $L$ means $U\le L$ and there are of course at most $\smax_q(l,d;k)$ $k$-subspaces contained in $L$. If 
$|k-l|<d$ and $L\in\mathcal{C}$, then no $k$-subspace contained in $L$ can be a codeword due to the minimum subspace distance. For 
$k>l$ the condition $U$ incident with $L$ means $L\le U$ and a similar reasoning applies. However, many $(k,l,a)$ triples lead to redundant 
constraints. Clearly if $a<a'$, then Inequality~(\ref{ie_add}) for $(k,l,a)$ implies Inequality~(\ref{ie_add}) for $(k,l,a')$. Moreover, 
if $\smax_q(l,d;k)=1$ for $k\le l$ or $\smax_q(v-l,d,k-l)=1$ for $k>l$, then Inequality~(\ref{ie_add}) for $(k,l,a)$ is implied 
by the constraints from the ILP of Proposition~\ref{prop_ILP1}. Removing these redundancies, we obtain the following more manageable 
set of additional constraints:
\begin{lemma}
  \label{lemma_additional_inequalities}
  Let $S\subseteq \mathbb{N}^3$ be given by
  $S=\left\{(k,l,\smax_q(l,d;k)) \,:\, 1\le k\le l\le v-1, \smax_q(l,d;k)\ge 2\right\}$
  $\cup$ $\left\{(k,l,\smax_q(v-l,d;k-l)) \,:\, 1\le l< k\le v-1, \smax_q(v-l,d;k-l)\ge 2\right\}$.
  If we add the inequalities from (\ref{ie_add}) for all $(k,l,a)\in S$ to the ILP formulation of Proposition~\ref{prop_ILP1},  
  we obtain an ILP formulation for the determination of $\smax_q(v,d;T)$, where any further inequality 
  from (\ref{ie_add}) is redundant.
\end{lemma}    

While the number of constraints in the ILP of Lemma~\ref{lemma_additional_inequalities} is larger than in the ILP of Proposition~\ref{prop_ILP1}, 
solution times typically are smaller. Any known upper bound $\smax_q(v,d;T')\le\Lambda$ for $T'\subseteq T$ can of course 
be incorporated by $\sum_{i\in T'}\delta_i\le \Lambda$. If $\smax_q(l,d;k)$ or $\smax_q(v-l,d;k-l)$ is not known, 
then we may take the smallest known upper bound as value of $a$ in the specification of $S$. 

As an example we consider the case $\smax_2(6,3)$. Here, the ILP of Proposition~\ref{prop_ILP1} reads:  
\begin{align*}
  \max \sum_{k=0}^6 \delta_k &&& \text{subject to}\\
    x_W+\sum_{U\in\G{6}{1}{2}} x_U &\le 1 &&\forall W\in\G{6}{0}{2}\\
    \sum_{U\in\G{6}{0}{2}\,:\,U\le W} x_U\,+\,x_W+ \sum_{U\in\G{6}{2}{2}\,:\,W\le U} x_U &\le1&&\forall W\in\G{6}{1}{2}\\
    \sum_{U\in\G{6}{1}{2}\,:\,U\le W} x_U\,+\,x_W\,+\,\sum_{U\in\G{6}{3}{2}\,:\,W\le U} x_U &\le1&&\forall W\in\G{6}{2}{2}\\    
    \sum_{U\in\G{6}{2}{2}\,:\,U\le W} x_U\,+\,x_W\,+\,\sum_{U\in\G{6}{4}{2}\,:\,W\le U} x_U &\le1&&\forall W\in\G{6}{3}{2}\\
    \sum_{U\in\G{6}{3}{2}\,:\,U\le W} x_U\,+\,x_W\,+\,\sum_{U\in\G{6}{5}{2}\,:\,W\le U} x_U &\le1&&\forall W\in\G{6}{4}{2}\\
    \sum_{U\in\G{6}{4}{2}\,:\,U\le W} x_U\,+\,x_W\,+\,\sum_{U\in\G{6}{6}{2}\,:\,W\le U} x_U &\le1&&\forall W\in\G{6}{5}{2}\\
    \sum_{U\in\G{6}{5}{2}\,:\,U\le W} x_U\,+\,x_W\ &\le 1&&\forall W\in\G{6}{6}{2}\\
    \sum_{U\in\G{6}{k}{2}} x_U &= \delta_k&&\forall 0\le k\le 6\\
    x_U&\in \{0,1\}&&\forall U\le\F_2^6.
\end{align*}

Using the exact values of $\smax_2(v,3;k)$ for $v\le 5$, the set $S$ of Lemma~\ref{lemma_additional_inequalities} 
can be chosen as
$$
  S=\left\{(2,4,5), (2,5,9), (3,1,9), (3,5,9), (4,1,9), (4,2,5)\right\}
$$ 
and the additional inequalities are given by 
\begin{eqnarray*}
  \sum_{K\in\G{6}{2}{2}\,:\, K\le L} x_K\,+\, 5x_L \le 5\quad \forall L\in\G{6}{4}{2}\\
  \sum_{K\in\G{6}{2}{2}\,:\, K\le L} x_K\phantom{\,+\, 9x_L}\ \le 9\quad \forall L\in\G{6}{5}{2}\\
  \sum_{K\in\G{6}{3}{2}\,:\, L\le K} x_K\,+\, 9x_L \le 9\quad \forall L\in\G{6}{1}{2}\\
  \sum_{K\in\G{6}{3}{2}\,:\, K\le L} x_K\,+\, 9x_L \le 9\quad \forall L\in\G{6}{5}{2}\\
  \sum_{K\in\G{6}{3}{2}\,:\, L\le K} x_K\phantom{\,+\, 9x_L} \le 9\quad \forall L\in\G{6}{1}{2}\\
  \sum_{K\in\G{6}{4}{2}\,:\, L\le K} x_K\,+\, 5x_L \le 5\quad \forall L\in\G{6}{2}{2}.
\end{eqnarray*}

We remark that the stated inequalities of Proposition~\ref{prop_ILP1} are also valid for even $d$. However, subspaces with subspace distance 
$d-1$ need to be excluded by additional inequalities. To this end let $a$, $b$, and $i$ be non-negative integers 
with $a+b-2i=d-1$ and $a<b$. With $\Lambda=A_q(v-i,d;b-i)$ the constraints
$$
  \sum_{U\in\G{v}{a}{q}\,:\,W\le U} \Lambda x_U\,+\,\sum_{U\in\G{v}{b}{q}\,:\,W\le U} x_U \le\Lambda \quad\forall W\in\G{v}{i}{q}
$$ 
prevent the existence of an $a$-subspace $A$ and $b$-subspace $B$ with $A\cap B=W$ and $A,B\in\mathcal{C}$. Here, $\Lambda$ may 
be replaced by any known upper bound for $A_q(v-i,d;b-i)$. We remark that there are also other inequalities with all coefficients 
being equal to $1$ achieving the same goal. However, those inequalities are more numerous and out of the scope of this paper.

A common method to reduce the complexity of the problem of the determination of $A_q(v,d)$, see \cite{kohnert2008construction} for 
the constant-dimension case, is to prescribe automorphisms, which results in lower bounds. For each automorphism $\varphi$ we can 
add the constraints $x_U=x_{\varphi(U)}$ for 
all $U\le\F_q^v$, which reduces the number of variables. Due to the group structure, several inequalities will also become identical, 
so that their number is also reduced. This approach is also called Kramer-Mesner method and the reduced system can be easily stated 
explicitly, see \cite{kohnert2008construction} for the constant-dimension case. Of course the optimal solution of the corresponding 
ILP gives just a lower bound, which however matches the upper bound for some instances and well chosen groups of automorphisms. An 
example for the prescription of automorphisms is given in Section~\ref{sec_dim_6}.   

\section{Subspace codes of very small dimensions and general results}
\label{sec_dim_small}

For each dimension $v$ of the ambient space the possible minimum subspace distances are contained in $\{1,2,\dots, v\}$. Arguably, 
the easiest case is that of minimum subspace distance $d=1$. Since $\sdist(X,Y)\ge 1$ for any two different subspaces of $\F_q^v$, we have 
$\smax_q(v,1)=\sum_{i=0}^v \gauss{v}{i}{q}$ and there is a unique isomorphism type, i.e., taking all subspaces. The quantity $\smax_q(v,2)$ 
was determined in \cite{ahlswede2009error}. For even $v$ we have $\smax_q(v,2)=\sum_{\underset{i\equiv 0\mod 2}{0\le i\le v}} \gauss{v}{i}{q}$ 
and for odd $v$ we have $\smax_q(v,2)=\sum_{\underset{i\equiv 0\mod 2}{0\le i\le v}} \gauss{v}{i}{q}=\sum_{\underset{i\equiv 1\mod 2}{0\le i\le v}} 
\gauss{v}{i}{q}$. The proof is based on a more general result of Kleitman \cite{kleitman1975extremal} on finite posets with the so-called LYM 
property. Matching examples are given by the sets of all subsets of $\F_q^v$ with either all even or all odd dimensions. The fact that these 
examples give indeed all isomorphism types was proven in \cite[Theorem 3.4]{honold2016constructions}. 

Next we look at large values for $d$. 
The case $\smax_q(v,v)$ was treated in \cite[Theorem 3.1]{honold2016constructions}. If $v$ is odd, then $\smax_q(v,v)=2$ and there are 
exactly $(v+1)/2$ isomorphism types consisting of an $i$-subspace and a disjoint $(v-i)$-subspace for $0\le i\le(v-1)/2$. If $v=2m$ is even, then 
$\smax_q(v,v)=\smax_q(v,v;m)=q^{m}+1$. Indeed, all codewords have to have dimension $m$. The  number  of  isomorphism  classes  of 
such  spreads  or, equivalently,  the  number  of  equivalence  classes  of  translation  planes  of  order $q^{m}$ with kernel containing
$\F_q$ under the equivalence relation generated by isomorphism and transposition \cite{dembowski2012finite,johnson2007handbook}, is generally 
unknown (and astronomically large even for modest parameter sizes). For $v\in\{2,4,6\}$ the Segre spread, see \cite{segre1964teoria}, is 
unique up to isomorphism. This is pretty easy to check for $v<6$ and done for $v=6$ in \cite{hall1956uniqueness}. For $v=8$ there 
are $8$ isomorphism types, see \cite{dempwolff1983classification}.

The case $\smax_q(v,v-1)$ was treated in \cite[Theorem 3.2]{honold2016constructions}. If $v=2k$ is even, then 
$\smax_q(v,v-1)=\smax_q(v,v-1;k)=q^k+1$. Beside dimension $k$ the codewords can only have dimension $k-1$ or $k+1$, where the latter 
cases can occur at most one each and all combinations are possible. It is well known that a partial $k$-spread of cardinality $q^k-1$ can 
always be extended to a $k$-spread. Having the $k$-spreads at hand for $k\le 4$ the resulting isomorphism types for the subspace codes can be 
determined easily. For $v=4$ there are $4$ and for $v=6$ there are $5$ isomorphism types, see \cite[Section 4]{honold2016constructions}. 
For $v=8$ we have removed one or two solids from the $8$ isomorphism types of $4$-spreads and obtained $17$ and $34$ isomorphism types of 
cardinality $16$ and $15$, respectively. The extension with $3$-subspaces and $5$-subspaces gives $31$ non-isomorphic subspaces codes 
with dimension distributions $\delta=(0,0,0,1,16,0,0,0,0)$ or $\delta=(0,0,0,0,16,1,0,0,0)$, respectively, and $502$ non-isomorphic subspaces codes 
with dimension distributions $\delta=(0,0,0,1,15,1,0,0,0)$. So, in total we have $572$ non-isomorphic $(8,17,7)_2$ codes.  

If $v=2k+1\geq 5$ is odd then $\smax_q(v,v-1)=\smax_q(v,v-1;k)=q^{k+1}+1$. The dimension distributions realized by optimal subspace 
codes are $(\delta_{k-1},\delta_k,\delta_{k+1},\delta_{k+2})=(0,q^{k+1}+1,0,0)$, $(0,0,q^{k+1}+1,0)$, $(0,q^{k+1},1,0)$, $(0,1,q^{k+1},0)$, 
$(0,q^{k+1},0,1)$, and $(1,0,q^{k+1},0)$. It is necessary to exclude the case $v=3$, where $\smax_2(3,2)=8>5$. For $v=5$ there are exactly 
$4+4+1+1+2+2=14$ isomorphism types, see \cite[Subsection 4.3]{honold2016constructions}. For $v=7$ there are exactly 
$715+715+37+37+176+176=1856$ isomorphism types, see \cite{honold2016classification}.

The case $\smax_q(v,v-2)$ is even more involved and only partial results are known, see \cite[Theorem 3.3]{honold2016constructions}. 
If $v=2k\ge 8$ is even, then $\smax_q(v,v-2)=\smax_q(v,v-2;k)$. Moreover, $\smax_2(6,4)=\smax_2(6,4;3)=77$, see 
\cite{honold2015optimal}, and $\smax_2(8,6)=\smax_2(8,6;4)=257$, see \cite{heinlein2017classifying}, are known. The number of 
isomorphism types for $v=6$ and $v=8$ is treated in Section~\ref{sec_dim_6} and Section~\ref{sec_dim_8}, respectively. If 
$v=2k+1\ge 5$ is odd, then $\smax_q(v,v-2)\in\left\{2q^{k+1}+1,2q^{k+1}+2\right\}$, $\smax_q(5,3)=2q^3+2$, and $\smax_2(7,5)=2\cdot 2^4+2=34$.
For $v=5$ there are $48\,217$ isomorphism types, see \cite[Subsection 4.4]{honold2016constructions}. For $v=7$ there are $39$ 
isomorphism types, see \cite{honold2016classification}, each consisting of $17$ planes and $17$ solids.

\section{Subspace codes in $\F_2^6$}
\label{sec_dim_6}

For 
subspace codes in $\F_2^6$ the only cases not treated in Section~\ref{sec_dim_small} are minimum subspace distance $d=3$ and $d=4$. 
For the latter we want to give all details, since the inspection in \cite{honold2016constructions} is rather scarce. First we remark 
$\smax_2(6,4;3)=77$, see \cite{honold2015optimal}. We start with an observation, due to Thomas Honold, that might be of independent 
interest: 
\begin{lemma}
  \label{lemma_point_a_q_6_4_3}
  Let $\mathcal{C}$ be a set of planes in $\mathbb{F}_q^6$ with minimum subspace distance at least $4$. Let $r(P)$ denote the number 
  of codewords of $\mathcal{C}$ that contain point $P$ and $r(H)$ denote the number of codewords of $\mathcal{C}$ that are contained 
  in hyperplane $H$. For every point $P$ or hyperplane $H$ we have $\#\mathcal{C}\le q^3(q^3+1)+r(P)$ and $\#\mathcal{C}\le q^3(q^3+1)+r(H)$, 
  respectively. Moreover, $r(P)\le q^3+1$ and $r(H)\le q^3+1$.
\end{lemma}
\begin{proof}
  We have $r(P)\le A_q(5,4;2)=q^3+1$, so that we also conclude by orthogonality $r(H)\le q^3+1$. Let $a_2(H)$ denote the number of codewords 
  of $\mathcal{C}$ whose intersection with $H$ is $2$-dimensional. With this we have 
  $a_2(H)=\frac{1}{q^2}\sum_{P\not\le H} r(P)\le\frac{1}{q^2}\cdot q^5\cdot (q^3+1)=q^3(q^3+1)$ and 
  $\#\mathcal{C}=a_2(H)+r(H)\le q^3(q^3+1)+r(H)$. By orthogonality we also obtain $\#\mathcal{C}\le q^3(q^3+1)+r(P)$.
\end{proof}

\begin{proposition}
  We have $\smax_2(6,4)=77$ and there are exactly $5$ optimal isomorphism types each containing planes only.
\end{proposition}
\begin{proof}
  Let $\mathcal{C}$ be a subspace code in $\F_2^6$ with maximum cardinality for minimum subspace distance $d=4$. 
  From \cite{honold2015optimal} we know $\delta_3\le 77$. 
  Moreover, we have $\delta_0+\delta_1\le 1$, $\delta_6+\delta_5\le 1$, $21\delta_0+\delta_1+\delta_2\le \smax_2(6,4;2)=21$, 
  and $21\delta_6+\delta_5+\delta_4\le \smax_2(6,4;4)=\smax_2(6,4;2)=21$. If $\delta_0=1$, then $\delta_1=\delta_2=\delta_3=0$, so that 
  $\#\mathcal{C}\le 1+21<77$. Thus, $\delta_0=0$ and, by orthogonality, $\delta_6=0$. Again by orthogonality, we can assume 
  $\delta_1+\delta_2\ge \delta_4+\delta_5$. If $\delta_1=1$, then $\delta_2=0$ and Lemma~\ref{lemma_point_a_q_6_4_3} 
  gives $\delta_3\le 72$, so that $\#\mathcal{C}\le 1+72+1<77$. Similarly, if $\delta_2\ge 1$, then $\delta_1=0$ and 
  Lemma~\ref{lemma_point_a_q_6_4_3} implies $\delta_3\le 72$. If $1\le\delta_2\le 2$, then $\#\mathcal{C}\le 72+2\cdot 2<77$. 
  For the cases with $\delta_2\ge 3$ we observe that every point of $\F_2^6$ can be contained in at most $\smax_2(5,4;2)=9$ 
  planes in $\mathcal{C}$. Any two lines and each pair of a line and a plane in $\mathcal{C}$ have to be disjoint, so that 
  $\delta_3\le \left(63-3\delta_2\right)\cdot\frac{9}{7}$, since any plane contains $7$ points. Thus, $\#\mathcal{C}\le 2\delta_2+\delta_3
  \le  81 -\frac{13\delta_2}{7}\le 81 -\frac{39}{7}<77$. To sum up, $\#\mathcal{C}=77$ is only possible if all codewords are planes. 
  The corresponding five isomorphism types have been determined in~\cite{honold2015optimal}.  
\end{proof}

The case of minimum subspace distance $d=3$ is more involved and we can only present partial results. A bit more notation from 
the derivation of $\smax_2(6,4;3)=77$, see \cite{honold2015optimal}, is necessary. Let $\mathcal{C}_3$ be a constant-dimension 
code in $\F_2^6$ with codewords of dimension $3$ and minimum subspace distance $4$. A subset of $\mathcal{C}_3$ consisting of $9$ 
planes passing through a common point $P$ will be called a \emph{$9$-configuration}.
\begin{lemma}(\cite[Lemma 7]{honold2015optimal})
\label{lma:9conf}
If $\#\mathcal{C}_3 \geq 73$ then $\mathcal{C}_3$ contains a $9$-configuration.
\end{lemma}
There are four isomorphism types of $9$-configurations. A subset of $\mathcal{C}_3$ of size $17$ will be called a \emph{$17$-configuration} 
if it is the union of two $9$-con\-fi\-gu\-ra\-tions, i.e., a $17$-configuration corresponds to a pair of points $(P,P')$ of degree $9$ that 
is connected by a codeword in $\mathcal{C}_3$ containing $P$ and $P'$.
\begin{lemma}(\cite[Lemma 8]{honold2015optimal})
\label{lma:17conf}
If $\#\mathcal{C}_3 \geq 74$ then $\mathcal{C}_3$ contains a $17$-configuration.
\end{lemma}
There are $12\,770$ isomorphism types of $17$-configurations, see \cite[Lemma 9]{honold2015optimal}.
\begin{proposition}
  \label{prop_a_2_6_3}
  $108\le \smax_2(6,3)\le 117$
\end{proposition}
\begin{proof}
  Let $\mathcal{C}$ be a subspace code in $\F_2^6$ with minimum subspace distance $3$. We have $\delta_0,\delta_1,\delta_5,\delta_6\le 1$,  
  $21\delta_0+\delta_1+\delta_2\le 21$, 
  $\delta_3\le 77$, and $21\delta_6+\delta_5+\delta_4\le 21$. Based on the classification of the optimal $(6,77,4;3)_2$ codes 
  the bound $\#\mathcal{C}\le 95$ for $\delta_3=77$ was determined in \cite{honold2016constructions}, so that we can assume $\delta_3\le 76$, 
  which gives $\#\mathcal{C}\le 21+76+21=118$ in general. If $\delta_0=1$, then $\#\mathcal{C}\le 1+76+21=98$, so that we can directly 
  assume $\delta_0=0$ and $\delta_6=0$, due to orthogonality and the example of cardinality $104$ found in \cite{honold2016constructions}. 
  From Lemma~\ref{lemma_point_a_q_6_4_3} we conclude $\delta_3\le 72$ if $\delta_1=1$, 
  which can be written as $4\delta_1+\delta_3\le 76$. By orthogonality, we also have $4\delta_5+\delta_3\le 76$.
  
  Using these constraints together with the ILP displayed at the end of Subsection~\ref{subsec_opt_formulations} and a
  prescribed group 
  \begin{equation}
    \label{eq_group_9}
    \left\langle
    \left(\begin{smallmatrix}
      1 & 0 & 0 & 0 & 0 & 0\\
      0 & 1 & 0 & 0 & 0 & 0\\
      0 & 0 & 0 & 1 & 0 & 0\\
      0 & 0 & 1 & 1 & 0 & 0\\
      0 & 0 & 0 & 0 & 0 & 1\\
      0 & 0 & 0 & 0 & 1 & 1
    \end{smallmatrix}\right),
    \left(\begin{smallmatrix}
      0 & 1 & 0 & 0 & 0 & 0\\
      1 & 1 & 0 & 0 & 0 & 0\\
      0 & 0 & 1 & 0 & 1 & 0\\
      0 & 0 & 0 & 1 & 0 & 1\\
      0 & 0 & 1 & 0 & 0 & 0\\
      0 & 0 & 0 & 1 & 0 & 0
    \end{smallmatrix}\right)
    \right\rangle,  
  \end{equation} 
  which is a direct product of two cyclic groups of order $3$, solving the corresponding ILP to optimality (in $3$ seconds) 
  gives $\smax_2(6,3)\ge 108$. Moreover, with the prescribed automorphisms, there are exactly four isomorphism types of codes 
  of cardinality $108$. In all cases the full automorphism group is indeed the group in (\ref{eq_group_9}) and the dimension distribution 
  is $\delta=(0,0,18,72,18,0,0)$. Moreover, the group fixes exactly $3$ lines, which complement the $18$ lines of the code to a 
  line spread.  
  
  For the upper bound we also utilize an ILP formulation, of course without prescribed automorphisms. If $\delta_i=1$ for some $i\in\{1,5\}$, 
  then $\#\mathcal{C}\le 21+72+21=114$. If $\mathcal{C}$ does not contain a $17$-configuration, then $\#\mathcal{C}\le 21+73+21=115$. 
  So, we set $\delta_0=\delta_1=\delta_5=\delta_6=0$ and prescribe a $17$-configuration, i.e., we fix the corresponding $x_U$-variables to 
  one. Using these constraints together with the ILP displayed at the end of Subsection~\ref{subsec_opt_formulations} gives an ILP formulation. 
  In order to further reduce the complexity, we modify the target function to $2\delta_2+\delta_3$ and set $\delta_4=0$. Solving the 
  corresponding $12\,770$ ILPs gives $\smax_2(6,3)\le 117$, while some solutions with target value $117$, $\delta_2=21$, and $\delta_3=75$ 
  were found. The computations took roughly half a year CPU time in total. We remark that $\delta_3=77$ is achievable in only $393$ and 
  $\delta_3=76$ in only $1076+393$ out of the $12\,770$ cases.     
\end{proof}

We remark that by prescribing $17$-configurations one can try to classify all 
$(6,M,4;3)_2$ codes with $M\ge 74$. However, this 
problem is computationally demanding. First experiments have resulted in $491$ non-isomorphic inclusion-maximal codes of cardinality $75$ and 
$88$ non-isomorphic inclusion maximal codes of cardinality $76$, i.e., there is no extendability result for the 
$(6,77,4;3)_2$ codes. 
Out of the $12\,770$ $17$-configurations only $2218$ allow a code of cardinality at least $75$.

Since the possible automorphism groups in $\F_2^6$ are quite manageable, we have checked all of them along the lines of the first part of 
the proof of Proposition~\ref{prop_a_2_6_3}. Note that if a group $G$ of automorphisms allows only code sizes smaller than $108$, then 
the same is true for all overgroups, which significantly reduces the computational effort. Moreover, at most one subgroup for each conjugacy  
class needs to be considered. After several hours of ILP computations, besides 
the identity group and the group in (\ref{eq_group_9}), there remain only the following groups that may 
allow code sizes of cardinality at least $108$:
\begin{itemize}
\item
$\left\langle\left(\begin{smallmatrix}
1 & 1 & 0 & 0 & 0 & 0\\
0 & 1 & 0 & 0 & 0 & 0\\
0 & 0 & 1 & 1 & 0 & 0\\
0 & 0 & 0 & 1 & 0 & 0\\
0 & 0 & 0 & 0 & 1 & 1\\
0 & 0 & 0 & 0 & 0 & 1
\end{smallmatrix}\right)\right\rangle\simeq\mathbb{Z}_2$, $\#\mathcal{C}\le 116$ 
\item
$\left\langle\left(\begin{smallmatrix}
0 & 1 & 0 & 0 & 0 & 0\\
1 & 1 & 0 & 0 & 0 & 0\\
0 & 0 & 0 & 1 & 0 & 0\\
0 & 0 & 1 & 1 & 0 & 0\\
0 & 0 & 0 & 0 & 0 & 1\\
0 & 0 & 0 & 0 & 1 & 1
\end{smallmatrix}\right)\right\rangle\simeq\mathbb{Z}_3$, $\#\mathcal{C}\le 115$ 
\item
$\left\langle\left(\begin{smallmatrix}
1 & 0 & 0 & 0 & 0 & 0\\
0 & 1 & 0 & 0 & 0 & 0\\
0 & 0 & 0 & 1 & 0 & 0\\
0 & 0 & 1 & 1 & 0 & 0\\
0 & 0 & 0 & 0 & 0 & 1\\
0 & 0 & 0 & 0 & 1 & 1
\end{smallmatrix}\right)\right\rangle\simeq\mathbb{Z}_3$, $\#\mathcal{C}\le 116$ 
\end{itemize}
where $\mathbb{Z}_p$ denotes a cyclic group of order $p$. The two latter groups are two non-isomorphic subgroups 
of the group of order nine in (\ref{eq_group_9}) that permit a code of size $108$. 

From the ILP computations in the proof of Proposition~\ref{prop_a_2_6_3} we conclude:
\begin{corollary}
  $\smax_2(6,3;\{2,3\})=\smax_2(6,3;\{0,1,2,3\})=96$
\end{corollary}
We remark that $\delta_3=77$ implies $\delta_0+\delta_1+\delta_2+\delta_3\le 88$, see \cite{honold2016constructions}.

\section{Subspace codes in $\F_2^7$}
\label{sec_dim_7}

For binary subspace codes in $\F_2^7$ the only cases not treated in Section~\ref{sec_dim_small} are minimum subspace distance $d=3$ and $d=4$. 
The upper bound $\smax_2(7,4)\le 407$ was obtained in \cite[Subsection 4.1]{honold2016constructions} using a quite involved and tailored analysis. 
The lower bound $\smax_2(7,4)\ge 334$ can be obtained taking $333$ planes in $\F_2^7$ with minimum subspace distance $4$, 
see \cite{heinlein2017subspace}, and adding the full ambient space $\F_2^7$ as a codeword. Here also the constant-dimension case is widely open, 
i.e, the best know bounds are $333\le\smax_2(7,4;3)\le 381$. If cardinality $381$ can be attained, then the code can have at most 
two automorphisms, see \cite{kiermaier2018order}. The upper bound $\smax_2(7,3)\le 776$ was obtained in \cite{bachoc2013bounds} using 
semidefinite programming. A construction for $\smax_2(7,3)\ge 584$ was described in \cite{etzion2013codes}. Based on extending 
a constant-dimension code of $329$ planes in $\F_2^7$ with minimum subspace distance~$4$ from \cite{honold2016putative}, the lower bound 
was improved to $\smax_2(7,3)\ge 593$ in \cite{honold2016constructions}.  
We remark that the group of order $16$ from \cite{heinlein2017subspace} gives a mixed dimension code of cardinality $574$ and ILP computation 
verifies that the group permits no subspace code of cardinality larger than $611$. Another construction of $329$  
planes with minimum subspace distance $4$ is stated in \cite{braun2014q} and based on the prescription of a cyclic group of order $15$ as 
automorphisms. Taking the same group and searching for planes and solids with subspace distance $3$ by an ILP formulation, we found a code 
of size $612$ consisting of $306$ planes and $306$ solids. Adding the empty and the full ambient space as codewords gives $\smax_2(7,3)\ge 614$.  
The ILP upper bound for subspace codes permitting this group is $713$ aborting after two days.

 
\section{Subspace codes in $\F_2^8$}
\label{sec_dim_8}

For binary subspace codes in $\F_2^8$ the cases not treated in Section~\ref{sec_dim_small} are given by minimum subspace distance 
$d=3$, $d=4$, $d=5$, and $d=6$. Here we will present some improvements. However, our results will only be partial in most cases. 
Even for constant-dimension codes the bounds $4801\le \smax_2(8,4;4)\le 6477$ is the current state of the art. The lower bound 
was found in \cite{braun2016new}. Extended by the empty and the full ambient space as codewords, this gives $\smax_2(8,4)\ge 4803$.
For $d=3$, the so-called Echelon-Ferrers construction, see \cite{etzion2009error}, gives $\smax_2(8,3)\ge 4907$, which was the 
previously best known lower bound. Using an integer linear programming formulation we found $857$ planes and $29$ $5$-subspaces that 
are compatible with the $4801$ solids with respect to subspace distance $d=3$, so that $\smax_2(8,3)\ge 5687$. We remark that the 
corresponding distance distribution is given by $3^{92028}, 4^{1333070}, 5^{1462022}, 6^{6719747}, 7^{2699636}, 8^{3861638}$. The 
upper bound $\smax_2(8,3)\le 9268$, see \cite{bachoc2013bounds}, of course is also valid for $\smax_2(8,4)$. However, for $d=4$ 
the upper bound can be significantly improved.

\begin{proposition}
  $\smax_2(8,4)\le 6479$
\end{proposition}  
\begin{proof}
  Let $\mathcal{C}$ be a subspace code in $\F_2^8$ with maximum 
  cardinality for minimum subspace distance $d=4$. 
  Setting $\delta_i=\delta_i(\mathcal{C})$ for $0\le i\le 8$, we observe $\delta_0,\delta_8\le 1$ and $\delta_4\le\smax_2(8,4;4)\le 6477$. 
  Due to orthogonality we assume $\delta_0+\delta_1+\delta_2+\delta_3\ge\delta_5+\delta_6+\delta_7+\delta_8$. If $\delta_0=1$, then 
  $\delta_1=\delta_2=\delta_3=0$ and $\#\mathcal{C}\le 2\cdot\left(\delta_0+\delta_1+\delta_2+\delta_3\right)+\delta_4\le 2+6477=6479$. 
  
  In the following we assume that $\mathcal{C}$ does not contain the empty space as a codeword. By $\mathcal{C}'$ we denote 
  the codewords of $\mathcal{C}$ that have a dimension of at most $4$, i.e., $\#\mathcal{C}'=\delta_1+\delta_2+\delta_3+\delta_4$ 
  and $\#\mathcal{C}\le 2\cdot\left(\delta_1+\delta_2+\delta_3\right)+\delta_4$. For each point $P$ in $\F_2^8$ we denote the set of 
  codewords from $\mathcal{C}'$ that contain $P$ by $\mathcal{C}'_P$. As abbreviation we use $\Delta_i=\delta_i(\mathcal{C}'_P)$ for 
  $1\le i\le 4$. If $\Delta_1\ge 1$, then $\Delta_1=1$ and $\Delta_2=\Delta_3=\Delta_4=0$. If $\Delta_2\ge 1$, then $\Delta_2=1$ and 
  $\Delta_3=0$. Since $\F_2^8/P\cong \F_2^7$ modding out $P$ from the codewords of $\mathcal{C}'_P$ gives a point $Q\le\F_2^7$, that 
  corresponds to the unique line in $\mathcal{C}'_P$, and $\Delta_4$ planes with minimum subspace distance $4$. Every line containing $Q$ 
  cannot be contained in a $3$-dimensional codeword. Since each plane contains $\gauss{3}{2}{2}=7$ lines, there are $\gauss{7}{2}{2}$ lines 
  in total, and $\gauss{6}{1}{2}$ lines containing $Q$, we have $\Delta_4\le \left(\gauss{7}{2}{2}-\gauss{6}{1}{2}\right)/7=372$. If $\Delta_4=372$, 
  then in any hyperplane $H$ of $\F_2^7$ not containing $Q$ the $\gauss{6}{2}{2}=651$ lines are covered by the, say, $x$ $3$-dimensional 
  codewords contained in $H$ and the $372-x$ lines of $3$-dimensional codewords not contained in $H$. Thus, $651=7\cdot x+1\cdot(372-x)
  =372+6x$, which has no integer solution, so that $\Delta_4\le 371$. In the remaining cases we have $\Delta_1=\Delta_2=0$. 
  Again we mod out $P$ and consider $\Delta_3$ lines and $\Delta_4$ planes in $\F_2^7$ with minimum subspace distance $d=4$. Since every 
  two lines and each pair of a line and a plane are disjoint we consider how the $127$ points of $\F_2^7$ are covered by the codewords. 
  Since any point can be contained in at most $21$ planes, we have $\Delta_4\le \left(127-3\Delta_3\right)\cdot\frac{21}{7}=381-9\Delta_3$.
  
  Summing $\Delta_i(P)$ over all points $P$ of $\F_2^8$ gives $\gauss{i}{1}{2}\cdot \delta_i$ for all $1\le i\le 4$, so that we consider 
  a \textit{score} 
  $$
    s(P)=\Delta_4(P)/\gauss{4}{1}{2} + 2\cdot \sum_{i=1}^3 \Delta_i(P)/\gauss{i}{1}{2}.
  $$     
  Summing over the scores of all points then gives an upper bound for $\#\mathcal{C}$ due to 
  $\#\mathcal{C}\le\delta_4+2\cdot\left(\delta_1+\delta_2+\delta_3\right)$. If $\Delta_1(P)=1$, then $s(P)=2$. If $\Delta_2(P)=1$, then 
  $s(P)\le \frac{2}{3}+\frac{371}{15}=\frac{381}{15}$. If $\Delta_3(P)\ge 1$, then 
  $$
    s(P)\le \frac{2\Delta_3(P)}{7}+\frac{381-9\Delta_3(P)}{15}=\frac{381}{15}-\frac{11\Delta_3(P)}{35}\le \frac{381}{15}.
  $$
  Since all scores are at most $\frac{381}{15}$, we have $\#\mathcal{C}\le 255\cdot \frac{381}{15}=6477<6479$.
\end{proof}       

\begin{proposition}
  \label{prop_a_2_8_5}
  $263\le \smax_2(8,5)\le 326$
\end{proposition}
\begin{proof}
  We have $\smax_2(8,5;\{0,1,2\})=\smax_2(8,5;\{6,7,8\})\le 1$, $\smax_2(8,5;3)=\smax_2(8,5;5)=\smax_2(8,6;3)=34$, and 
  $\smax_2(8,5;4)=\smax_2(8,6;4)\le 257$, so that $A_2(8,5)\le 1+34+257+34+1=327$. Next, we use the classification of 
  the optimal codes attaining $A_2(8,5;4)=A_2(8,6;4)=257$. For both types 
  removing a suitable codeword gives a $4$-dimensional subspace $F$ that has empty intersection with all $256$ remaining codewords, i.e., 
  which is a lifted MRD code. Assume that $\mathcal{C}$ contains such a lifted MRD code. All points that are not contained in $F$ are contained 
  in exactly $16$ codewords and every line in $\F_2^8$  with empty intersection with $F$ is contained in exactly one codeword. Since planes 
  and solids can share at most one point, a $2$-dimensional codeword has to be disjoint to the $4$-dimensional codewords, and 
  $\smax_2(4,4;2)=5$, we have $\delta_2(\mathcal{C})+\delta_3(\mathcal{C})\le 5$, $\delta_5(\mathcal{C})+\delta_6(\mathcal{C})\le 5$, and 
  $\delta_4(\mathcal{C})\le 257$, so that $\#\mathcal{C}\le 267$. Otherwise at most $256$ solids can be contained in the code, which gives 
  the desired upper bound.
  
  For the lower bound we use the Echelon-Ferrers construction, see \cite{etzion2009error}, with pivot vectors 
  $(1,1,1,1,0,0,0,0)$, $(0,1,0,0,1,1,0,0)$, $(1,0,1,0,1,0,1,1)$, $(0,0,0,1,0,1,1,1)$. Corresponding rank distance codes 
  of cardinalities $256$, $4$, $2$, and $1$, respectively, can be constructed using the methods from \cite{etzion2016optimal}.
\end{proof}

Prescribing the lifted MRD code gives an upper bound of $263$ via an ILP computation. There is also a code of 
cardinality $263$ with dimension distribution $\delta=(0,0,0,3,257,3,0,0,0)$.


We continue with subspace distance $d=6$. 
\cite[Theorem 3.3(i)]{honold2016constructions} implies $\smax_2(8,6)=\smax_2(8,6;4)$, where the latter was 
determined to $\smax_2(8,6;4)=257$ in \cite{heinlein2017classifying}. For the constant-dimension case all 
two isomorphism types were classified in \cite[Theorem 1]{heinlein2017classifying}. An essential building block 
is the lifted Gabidulin code $\mathcal{G}_{8,4,3}$, which corresponds to the unique $4\times 4$ MRD code with rank distance $3$ over 
$\mathbb{F}_2$, see \cite[Theorem 10]{heinlein2017classifying}. It can be extended by a solid in exactly 
two non-isomorphic ways. By combining the approaches of both papers we can classify the isomorphism types in the mixed-dimension 
case:

\begin{theorem}
  \label{thm_2_8_6_classification}
  There are exactly eight non-isomorphic subspace codes attaining the maximum cardinality $\smax_2(8,6)=257$. 
  The eight isomorphism types are given by $\gabidulin_{8,4,3}$ extended by a codeword $U$ with 
  $\tau(U)$ given by
  \begin{itemize}
    \item $\left(\begin{smallmatrix}
          0&0&0&0&1&0&1&0\\
          0&0&0&0&0&1&0&1
          \end{smallmatrix}\right)
          $,
    \quad 
    $\left(\begin{smallmatrix}
          0&0&0&0&0&0&1&0\\
          0&0&0&0&0&0&0&1
          \end{smallmatrix}\right)$,\\[1mm]       
    \item $\left(\begin{smallmatrix}
          0&0&0&0&0&1&0&0\\
          0&0&0&0&0&0&1&0\\
          0&0&0&0&0&0&0&1
          \end{smallmatrix}\right)$,\\[1mm]
    \item $\left(\begin{smallmatrix}
          0&0&0&0&1&0&0&0\\
          0&0&0&0&0&1&0&0\\
          0&0&0&0&0&0&1&0\\
          0&0&0&0&0&0&0&1
          \end{smallmatrix}\right)$, 
    \quad
    $\left(\begin{smallmatrix}
          0&0&0&1&0&0&0&0\\
          0&0&0&0&1&0&0&0\\
          0&0&0&0&0&1&0&0\\
          0&0&0&0&0&0&1&0
          \end{smallmatrix}\right)$,\\[1mm]
    \item $\left(\begin{smallmatrix}
          0&0&0&1&0&0&0&0\\
          0&0&0&0&1&0&0&0\\
          0&0&0&0&0&1&0&0\\
          0&0&0&0&0&0&1&0\\
          0&0&0&0&0&0&0&1
          \end{smallmatrix}\right)$,\\[1mm]
    \item $\left(\begin{smallmatrix}
          0&0&1&0&0&0&0&0\\
          0&0&0&1&0&0&0&0\\
          0&0&0&0&1&0&0&0\\
          0&0&0&0&0&1&0&0\\
          0&0&0&0&0&0&1&0\\
          0&0&0&0&0&0&0&1
          \end{smallmatrix}\right)$, or 
    $\left(\begin{smallmatrix}
          1&0&0&1&0&0&0&0\\
          0&1&0&1&0&0&0&0\\
          0&0&0&0&1&0&0&0\\
          0&0&0&0&0&1&0&0\\
          0&0&0&0&0&0&1&0\\
          0&0&0&0&0&0&0&1
          \end{smallmatrix}\right)
          $.
  \end{itemize} 
\end{theorem}
\begin{proof}
  Let $\mathcal{C}$ be a subspace code in $\F_2^8$ with minimum subspace distance $d=6$ and maximum cardinality. 
  As discussed above, we have $\#\mathcal{C}=257$. Setting $\delta_i=\delta_i(\mathcal{C})$ for $0\le i\le 8$, we 
  observe $\delta_0+\delta_1+\delta_2\le 1$, $\delta_3\le \smax_2(8,6;3)=34$,  $\delta_4\le \smax_2(8,6;4)=257$, 
  $\delta_5\le \smax_2(8,6;5)=\smax_2(8,6;3)=34$, and $\delta_6+\delta_7+\delta_8\le 1$. If $\delta_0=1$ or $\delta_1=1$, then 
  $\delta_0+\delta_1=1$ and $\delta_2+\delta_3+\delta_4=0$, so that $\#\mathcal{C}\le 1+34+1<257$, which is a contradiction. 
  Thus $\delta_0=\delta_1=0$ and, by orthogonality, $\delta_7=\delta_8=0$. As an abbreviation we denote the set of $4$-dimensional 
  codewords by $\mathcal{C}_4$, i.e., $\delta_4=\#\mathcal{C}_4$.
  
  Again by orthogonality, we assume $\delta_2+\delta_3\ge\delta_5+\delta_6$. The case $\delta_2+\delta_3=0$ was treated in 
  \cite[Theorem 1]{heinlein2017classifying}. We first consider the cases where $\delta_3$ is \textit{large}. Note that $\delta_3\ge 1$ 
  implies $\delta_2=0$. For a given point $P$ let $\mathcal{C}_P=\left\{X\in\mathcal{C}_4\,:\,P\le X\right\}$ the set 
  of $4$-dimensional codewords containing $P$. Since two $4$-subspaces in $\mathcal{C}$ cannot intersect in a line, we conclude 
  $\#\mathcal{C}_P\le \smax_2(7,6;3)=17$ modding out $P$ from the codewords of $\mathcal{C}_P$. In $\mathcal{C}$ each two planes 
  and each pair of a plane and a solid have to be disjoint. Since every solid contains $15$ points we thus have 
  $\delta_4 \le\left(255-7\delta_3\right)\cdot\frac{17}{15}$, so that $\#\mathcal{C}\le \delta_4+2\delta_3\le 
  289-\frac{89\delta_3}{15}$ for $\delta_3\ge 1$. If $\delta_3\ge 6$ this gives $\#\mathcal{C}\le 253.4$, so that we can assume 
  $\delta_3\le 5$ in the following. 
  
  If $2\le \delta_3\le 5$, then let $X_1,X_2$ two distinct planes in $\mathcal{C}$. We consider the set $\mathcal{S}$ of 
  $6$-subspaces $S$ such that $\dim(S\cap X_1)=\dim(S\cap X_2)=1$. Either theoretically or computationally, we can easily 
  determine $\#\mathcal{S}=\gauss{3}{1}{2}^2\cdot 2^5\cdot (2^2-1)\cdot(2-1)=4704$. Each element $C\in\mathcal{C}_4$ 
  is contained in at least $\gauss{4}{2}{2}-2\cdot\gauss{3}{2}{2}=21$ elements from $\mathcal{S}$, which may be simply 
  checked by a computer enumeration or the following theoretic reasoning. There are $\gauss{4}{2}{2}$ 
  $6$-subspaces $S$ containing $C$. Since $X_1$ and $X_2$ are disjoint from $C$ we have $\dim(S\cap X_1),\dim(S\cap X_2)\in\{1,2\}$. 
  If $S$ contains any of the $\cdot\gauss{3}{2}{2}$ lines of $X_1$, then this line and $C$ uniquely determine $S$. The same 
  applies to $X_2$. Thus, $C$ is contained in at least $\gauss{4}{2}{2}-2\cdot\gauss{3}{2}{2}=21$ elements from $\mathcal{S}$.   
  Since no two $4$-dimensional codewords can be contained in an element from $\mathcal{S}$,  
  we have $\delta_4\le \frac{4704}{21} \le 234$ and $\#\mathcal{C}\le 2\delta_3+\delta_4\le 2\cdot 5+
  234<257$.
  
  Until now, we have concluded $\delta_2+\delta_3\le 1$, $\delta_5+\delta_6\le 1$, and $255\le \delta_4\le 257$. Next we consider 
  the cases $\delta_2=1$ or $\delta_3=1$. 
  
  If $\delta_2=1$, then let $L$ denote the unique $2$-dimensional codeword and $\mathcal{S}$ 
  be the set of $6$-subspaces in $\F_2^8$ that are disjoint to $L$. We have $\#\mathcal{S}=2^{12}=4096$ and each element from $\mathcal{C}_4$ 
  is contained in exactly $r=2^4=16$ elements from $\mathcal{S}$, so that $\delta_4\le \frac{4096}{16}=256$. Assume that $X$ is a 
  $5$-dimensional codeword. Then, $L$ and $X$ have to be disjoint and $X$ intersects every element from $\mathcal{C}_4$ exactly in a point. 
  So, if $X$ intersects an element $S\in\mathcal{S}$ in a solid, then $S$ cannot contain a $4$-dimensional codeword. Since there are 
  at least $\gauss{5}{4}{2}=31$ such elements in $\mathcal{S}$, we have $\delta_4\le \frac{4096-31}{16}<255$, so that $\#\mathcal{C}<257$ 
  and we assume that there is no $5$-dimensional codeword.  
  Alternatively, assume that $Y$ is a $6$-dimensional codeword. Then either $L$ and $Y$ are disjoint or they intersect in a point $P$. 
  In both cases $Y$ intersects every element from $\mathcal{C}_4$ in precisely a line. So, if $Y$ intersects an element $S\in\mathcal{S}$ in a 
  $5$-subspace, then $S$ cannot contain a $4$-dimensional codeword. If $\dim(L\cap Y)=0$, then there are at least $\gauss{6}{5}{2}=63$ such 
  elements in $\mathcal{S}$, so that $\delta_4\le \frac{4096-63}{16}<253$ and $\#\mathcal{C}<257$. If $L\cap Y=P$, then there are 
  at least $\gauss{6}{5}{2}-\gauss{5}{4}{2}=32$ such elements in $\mathcal{S}$, so that $\delta_4\le \frac{4096-32}{16}= 254$ and $\#\mathcal{C}<257$. 
  To sum up, if $\delta_2=1$, then only $\delta_3=\delta_5=\delta_6=0$ and $\delta_4=256$ is possible.
  
  If $\delta_3=1$, then let $E$ denote the unique $3$-dimensional codeword and $\mathcal{S}$ be the set of $6$-subspaces in $\F_2^8$ that 
  intersect $E$ in precisely a point. We have $\#\mathcal{S}=\gauss{3}{1}{2}\cdot 2^{10}=7168$ and each element from $\mathcal{C}_4$ is 
  contained in exactly $r=\gauss{3}{1}{2}\cdot 2^2=28$ elements from $\mathcal{S}$, so that $\delta_4\le \frac{7168}{28}=256$. Due 
  to the previous argument and orthogonality $\delta_6=1$ is impossible, so that we assume that $X$ is a $5$-dimensional codeword. Then, 
  $X$ has to intersect 
  every element from $\mathcal{C}_4$ in precisely a point and is either disjoint to $E$ or $P=X\cap E$ is a point. So, if $X$ intersects an 
  element $S\in\mathcal{S}$ in a solid, then $S$ cannot contain a $4$-dimensional codeword. If $\dim(E\cap X)=0$, then there are at least  
  $\gauss{3}{1}{2}\cdot\gauss{5}{4}{2}=217$ such elements in $\mathcal{S}$, so that $\delta_4\le \frac{28672-217}{112}<255$ and 
  $\#\mathcal{C}<257$. If $E\cap X=P$, then $X$ contains $16$ solids disjoint to $P$. Adding a point in $E\backslash P$ gives 
  $16\cdot 6=96$ choices for a $5$-subspace $F\le\langle E,X\rangle$ that intersects $E$ in a point. In each of these cases there are two 
  possibilities to extend $F$ to different elements in $\mathcal{S}$, so that $\delta_4\le \frac{28672-192}{112}<255$ and 
  $\#\mathcal{C}<257$. To sum up, if $\delta_3=1$, then only $\delta_2=\delta_5=\delta_6=0$ and $\delta_4=256$ is possible.
  
  So, it remains to consider the cases where $\delta_2+\delta_3=1$ and $\delta_4=256$. By $L$ or $E$ we denote the unique $2$- or 
  $3$-dimensional codeword in $\mathcal{C}$, respectively.  
  Every hyperplane $H$ of $\F_2^8$ contains at most $\smax_2(7,6;3)=17$ elements from $\mathcal{C}_4$ 
  due to orthogonality. In \cite{honold2016classification} the authors determined that there are exactly $715$ non-isomorphic 
  subspace codes in $\F_2^7$ with minimum subspace distance $d=6$ consisting of $17$ planes and $14\,445$ non-isomorphic such subspace codes 
  consisting of $16$ planes. By orthogonality, there are exactly $715$ non-isomorphic  subspace codes in $\F_2^7$ with minimum subspace 
  distance $d=6$ consisting of $17$ solids and $14\,445$ non-isomorphic such subspace codes consisting of $16$ solids. If a hyperplane $H$ 
  of $\F_2^8$ contains $17$ elements from $\mathcal{C}_4$, then each point is covered at least once, which also can 
  be directly deduced from the unique hole configuration of a maximum partial plane spread in $\F_2^7$, see e.g.\ \cite{honold2018partial}. 
  Since at least a point of $L$ or $E$ is contained in any hyperplane, we conclude that $H$ contains at most $16$ elements from $\mathcal{C}_4$. 
  Similarly, if $P$ is a point of $\mathbb{F}_2^8$ that is contained in $17$ elements from $\mathcal{C}_4$, then there exists a $5$-subspace $K$ 
  and a solid $S'\le K$ such that the points in $K\backslash S'$ are exactly those that are not contained in any element from $\mathcal{C}_4$. 
  This fact can again be verified computationally, using the $715$ non-isomorphic configurations mentioned above, or from the 
  unique hole configuration of a maximum partial plane spread in $\F_2^7$. Since neither $L$ nor $E$ can be contained in $K\backslash S'$, 
  every point $P$ of $\F_2^8$ is incident with at most $16$ elements in $\mathcal{C}_4$. If every hyperplane contains at most $15$ codewords, 
  then $\#\mathcal{C}_4\le \gauss{8}{1}{2}\cdot 15/\gauss{4}{1}{2}=255$, i.e., those cases do not need to be considered. Now, let $P$ be a 
  point and $H$ be a hyperplane of $\F_2^8$ such that $P$ is not contained in $H$. By $\mathcal{C}_{P,H}$ we denote the 
  set of codewords from $\mathcal{C}_4$ that either contain $P$ or are contained in $H$. If $\#\mathcal{C}_{P,H}=31$ we speak of a 
  $31$-configuration, see \cite{heinlein2017classifying}. As argued above, there is a hyperplane $H$ of $\F_2^8$ containing exactly $16$ 
  codewords from $\mathcal{C}_4$. If $\mathcal{C}_4$ contains at most $14$ codewords incident with a point $P$ for all points $P$ that are not 
  contained in $H$, then $15\cdot \mathcal{C}_4\le 127\cdot 16+128\cdot 14=3824=15\cdot 254+14$, so that $\#\mathcal{C}_4\le 254$. 
  Thus, $\mathcal{C}_4$ contains  a $31$-configuration $\mathcal{C}_{P,H}$ with $16$ solids in $H$ and $15$ solids containing $P$.
    
  In \cite{heinlein2017classifying} the $715+14445$ isomorphism types of subspace codes in $\F_2^7$ with minimum subspace distance $6$ 
  consisting of $17$ or $16$ solids were tried to extend to subspace codes of solids in $\F_2^8$ with minimum subspace distance $6$. As an 
  intermediate step $31$-configurations were enumerated. So, all required preconditions are also given in our situation. For the exclusion 
  of cases linear and integer linear programming formulations have been used. Let $z_{8}^{\text{LP}}$, see Lemma~\ref{lemma_ILP_1}, 
  and $z_{7}^{\text{BLP}}$, see Lemma~\ref{lemma_ILP_2},  
  denote the optimal target values of two of such (integer) linear programming formulations used and defined in \cite{heinlein2017classifying}. 
  According to \cite[Table 2]{heinlein2017classifying}, see also Table~\ref{tab:details}, only the cases of hyperplanes with indices $1$, and $9$ can lead to $256$ solids. 
  We state that there are four additional cases in which $256\le z_{8}^{\text{LP}}<257$, which are not listed in 
  \cite[Table 2]{heinlein2017classifying}. 
  The next smaller LP value is $255.879$, which is significantly far away from $256$. In all those four cases $z_{7}^{\text{BLP}}\le 254$ 
  is verified in less than $2$~minutes. Applying ILP $z_{7}^{\text{BLP}}$ to the case with index $9$ yields an upper bound of 
  $255$ after $75$~hours of computation time. Thus, there just remains the hyperplane with index $1$.
    
  Extending the hyperplane with index $1$ to a $31$-configuration yields $z_8^{\text{LP}}\ge 255.9$ in $234$ cases and 
  then $z_8^{\text{BLP}}\ge 256$ in $2$ cases. 
  In the latter cases there exists a unique solid $S$ that is disjoint 
  to the $31$ codewords.
  
  Next we use ILP computations in order to obtain some structural results. For any line $L$ disjoint to the $31$ codewords 
  and not completely contained in $S$, we try to enlarge each given $31$-configuration keeping $L$ disjoint to all codewords. 
  It turned out, that at most $255$ $4$-dimensional codewords are possible in total. Next we forbid that $4$-dimensional codewords 
  intersect in a dimension larger than $2$ and check all cases in which a $4$-dimensional codeword intersects $S$ in dimension $1$ or $2$. 
  In all cases we verify by an ILP computation that then the total number of $4$-dimensional codewords is at most $255$. 
    
  Thus, we have computationally verified, that $\delta_4=256$ is only possible if $S$ is disjoint to every codeword. This means, that  
  $\mathcal{C}_4$ is a lifted MRD code. From \cite[Theorem 10]{heinlein2017classifying} we then conclude that the code is isomorphic 
  to the lifted version of $\gabidulin_{8,4,3}$. Since all points outside of $S$ are covered, $L$ or $E$ have to lie completely 
  in $S$. The automorphism group of the lifted version of $\gabidulin_{8,4,3}$ has order $230\,400$ and acts transitively on 
  the set of points or planes that are contained in $S$ and on the set of $5$-subspaces that contain $S$. For lines there are two orbits, 
  one of length~$5$ and one of length~$30$. Thus, we obtain the isomorphism types listed in the statement of the theorem.
\end{proof}

Taking the automorphism $\pi$ into account reduces the $8$ isomorphism types of Theorem~\ref{thm_2_8_6_classification} to $5$. We remark 
that an alternative approach for the last part of the proof of Theorem~\ref{thm_2_8_6_classification} would be to classify 
all constant-dimension codes in $\F_2^8$ with minimum subspace distance $d=6$ consisting of $256$ solids and to extend them 
to subspace codes. Given the computational results from \cite{heinlein2017classifying} and the proof of 
Theorem~\ref{thm_2_8_6_classification} there remain only the types of hyperplanes with indices $1$, $7$, and $8$. This approach is 
indeed feasible but rather involved, so that we present a sketch in Section~\ref{sec_256} in the appendix as a further justification of 
the correctness of Theorem~\ref{thm_2_8_6_classification}. As a by-product we obtain the information that all $(8,256,6;4)_2$ 
constant-dimension codes can be extended to an $(8,257,6;4)_2$ constant-dimension code, see Proposition~\ref{prop_256}. This is different 
to the situation of $A_2(6,4;3)=77$, where there is no extendability result for $(6,76,4;3)_2$ constant-dimension codes. 

\begin{table}[htbp]
  \centering
  \(
  \begin{array}{|c||c|c|c|c|c|c|c|c|}
    \hline
    v\backslash d&1&2&3&4&5&6&7&8\\\hline\hline
    1&2(1)&\multicolumn{1}{c}{}&\multicolumn{1}{c}{}&\multicolumn{1}{c}{}&\multicolumn{1}{c}{}&\multicolumn{1}{c}{}&\multicolumn{1}{c}{}&\\\cline{1-3}
    2&5(1)&3(1)&\multicolumn{1}{c}{}&\multicolumn{1}{c}{}&\multicolumn{1}{c}{}&\multicolumn{1}{c}{}&\multicolumn{1}{c}{}&\\\cline{1-4}
    3&16(1)&8(2)&2(2)&\multicolumn{1}{c}{}&\multicolumn{1}{c}{}&\multicolumn{1}{c}{}&\multicolumn{1}{c}{}&\\\cline{1-5}
    4&67(1)&37(1)&5(4)&5(1)&\multicolumn{1}{c}{}&\multicolumn{1}{c}{}&\multicolumn{1}{c}{}&\\\cline{1-6}
    5&374(1)&187(2)&18(48217)&9(14)&2(3)&\multicolumn{1}{c}{}&\multicolumn{1}{c}{}&\\\cline{1-7}
    6&2825(1)&1521(1)&\text{108--117}&77(5)&9(5)&9(1)&\multicolumn{1}{c}{}&\\\cline{1-8}
    7&29212(1)&14606(2)&\text{614--776}&\text{334--407}&34(39)&17(1856)&2(4)&\\\cline{1-9}
    8&417199(1)&222379(2)& \text{5687--9268} & \text{4803--6479} & \text{263--326} & 257(8) & 17(572) & 17(8)\\\hline
  \end{array}
  \)\\[1ex]
  \caption{$\smax_2(v,d)\label{tbl:smax2}$ and isomorphism types of optimal codes for $v\leq 8$.}
  \label{table_overview} 
\end{table}

\section{Conclusion}
\label{sec_conclusion}

In this paper we have considered the determination of the maximal code sizes $\smax_2(v,d)$ of binary subspace codes in $\F_2^v$ for 
all $v\le 8$ and all minimum distances $d\in\{1,\dots,v\}$. The precise numbers are completely known for $v\le 5$ only. For larger 
dimensions of the ambient space we have obtained some improvements of lower and upper bounds. Whenever $\smax_2(v,d)$ is known 
exactly the number of isomorphism types is known, where we add two classifications in this paper. The current state-of-the-art is 
summarized in Table~\ref{table_overview}. The numbers in brackets state the number of isomorphism types with respect to $\GL(v,\F_2)$.

Of course, a natural challenge is to further tighten the stated bounds. Especially for minimum subspace distance $d=3$ it 
would be very interesting to come up with improved general constructions. The smallest open case is still $\smax_2(6,3)$. 


\begin{thebibliography}{10}

\bibitem{ahlswede2009error}
Rudolf Ahlswede and Harout Aydinian, \emph{On error control codes for random
  network coding}, Network Coding, Theory, and Applications, 2009. NetCod'09.
  Workshop on, IEEE, 2009, pp.~68--73.

\bibitem{bachoc2013bounds}
Christine Bachoc, Alberto Passuello, and Frank Vallentin, \emph{Bounds for
  projective codes from semidefinite programming}, Advances in Mathematics of
  Communications \textbf{7} (2013), no.~2, 127--145.

\bibitem{braun2016new}
Michael Braun, Patric~R.J. {\"O}sterg{\aa}rd, and Alfred Wassermann, \emph{New
  lower bounds for binary constant-dimension subspace codes}, Experimental
  Mathematics (2016), 1--5.

\bibitem{braun2014q}
Michael Braun and Jan Reichelt, \emph{$q$-analogs of packing designs}, Journal
  of Combinatorial Designs \textbf{22} (2014), no.~7, 306--321.

\bibitem{cai2006network}
Ning Cai and Raymond~W. Yeung, \emph{Network error correction, {I}{I}: Lower
  bounds}, Communications in Information \& Systems \textbf{6} (2006), no.~1,
  37--54.

\bibitem{delsarte1978bilinear}
Philippe Delsarte, \emph{Bilinear forms over a finite field, with applications
  to coding theory}, Journal of Combinatorial Theory, Series A \textbf{25}
  (1978), no.~3, 226--241.

\bibitem{dembowski2012finite}
Peter Dembowski, \emph{Finite geometries: Reprint of the 1968 edition},
  Springer Science \& Business Media, 2012.

\bibitem{dempwolff1983classification}
Ulrich Dempwolff and Arthur Reifart, \emph{The classification of the
  translation planes of order 16, {I}}, Geometriae Dedicata \textbf{15} (1983),
  no.~2, 137--153.

\bibitem{etzion2016optimal}
Tuvi Etzion, Elisa Gorla, Alberto Ravagnani, and Antonia Wachter-Zeh,
  \emph{Optimal {F}errers diagram rank-metric codes}, IEEE Transactions on
  Information Theory \textbf{62} (2016), no.~4, 1616--1630.

\bibitem{etzion2009error}
Tuvi Etzion and Natalia Silberstein, \emph{Error-correcting codes in projective
  spaces via rank-metric codes and {F}errers diagrams}, IEEE Transactions on
  Information Theory \textbf{55} (2009), no.~7, 2909--2919.

\bibitem{etzion2013codes}
Tuvi Etzion and Natalia Silberstein, \emph{Codes and designs related to lifted {M}{R}{D} codes}, IEEE
  Transactions on Information Theory \textbf{59} (2013), no.~2, 1004--1017.

\bibitem{etzion2016galois}
Tuvi Etzion and Leo Storme, \emph{Galois geometries and coding theory},
  Designs, Codes and Cryptography \textbf{78} (2016), no.~1, 311--350.

\bibitem{feulner2009automorphism}
Thomas Feulner, \emph{The automorphism groups of linear codes and canonical
  representatives of their semilinear isometry classes}, Advances in
  Mathematics of Communications \textbf{3} (2009), no.~4, 363--383.

\bibitem{ubt_epub42}
Thomas Feulner, \emph{{Eine kanonische Form zur Darstellung {\"a}quivalenter
  Codes : Computergest{\"u}tzte Berechnung und ihre Anwendung in der
  Codierungstheorie, Kryptographie und Geometrie}}, Ph.D. thesis, University of
  Bayreuth, Bayreuth, March 2014.

\bibitem{gabidulin1985theory}
Ernest~Mukhamedovich Gabidulin, \emph{Theory of codes with maximum rank
  distance}, Problemy Peredachi Informatsii \textbf{21} (1985), no.~1, 3--16.

\bibitem{greferath2018network}
Marcus Greferath, Mario~Osvin Pav{\v{c}}evi{\'c}, Natalia Silberstein, and
  Mar{\'\i}a~{\'A}ngeles V{\'a}zquez-Castro, \emph{Network coding and subspace
  designs}, Springer, 2018.

\bibitem{hall1956uniqueness}
Marshall Hall~Jr, J.~Dean Swift, and Robert~J. Walker, \emph{Uniqueness of the
  projective plane of order eight}, Mathematical Tables and Other Aids to
  Computation (1956), 186--194.

\bibitem{heinlein2017classifying}
Daniel Heinlein, Thomas Honold, Michael Kiermaier, Sascha Kurz, and Alfred
  Wassermann, \emph{Classifying optimal binary subspace codes of length 8,
  constant dimension 4 and minimum distance 6}, arXiv preprint 1711.06624
  (2017).

\bibitem{heinlein2016tables}
Daniel Heinlein, Michael Kiermaier, Sascha Kurz, and Alfred Wassermann,
  \emph{Tables of subspace codes}, arXiv preprint 1601.02864 (2016).

\bibitem{heinlein2017subspace}
Daniel Heinlein, Michael Kiermaier, Sascha Kurz, and Alfred Wassermann, 
\emph{A subspace code of size $333$ in the setting of a binary
  $q$-analog of the {F}ano plane}, arXiv preprint 1708.06224 (2017).

\bibitem{honold2016putative}
Thomas Honold and Michael Kiermaier, \emph{On putative $q$-analogues of the
  {F}ano plane and related combinatorial structures}, Dynamical Systems, Number
  Theory and Applications: A Festschrift in Honor of Armin Leutbecher's 80th
  Birthday, World Scientific, 2016, pp.~141--175.

\bibitem{honold2015optimal}
Thomas Honold, Michael Kiermaier, and Sascha Kurz, \emph{Optimal binary
  subspace codes of length $6$, constant dimension $3$ and minimum subspace
  distance $4$}, Topics in finite fields, Contemp. Math., vol. 632, Amer. Math.
  Soc., Providence, RI, 2015, pp.~157--176. \MR{3329980}

\bibitem{honold2016classification}
Thomas Honold, Michael Kiermaier, and Sascha Kurz, \emph{Classification of large partial plane spreads in ${P}{G}(6, 2)$
  and related combinatorial objects}, arXiv preprint 1606.07655 (2016).

\bibitem{honold2016constructions}
Thomas Honold, Michael Kiermaier, and Sascha Kurz, \emph{Constructions and bounds for mixed-dimension subspace codes},
  Advances in Mathematics of Communications \textbf{10} (2016), no.~3,
  649--682.

\bibitem{honold2018partial}
Thomas Honold, Michael Kiermaier, and Sascha Kurz, \emph{Partial spreads and vector space partitions}, Network Coding and
  Subspace Designs, Springer, 2018, pp.~131--170.

\bibitem{johnson2007handbook}
Norman Johnson, Vikram Jha, and Mauro Biliotti, \emph{Handbook of finite
  translation planes}, CRC Press, 2007.

\bibitem{kiermaier2017improvement}
Michael Kiermaier and Sascha Kurz, \emph{An improvement of the {J}ohnson bound
  for subspace codes}, arXiv preprint 1707.00650 (2017).

\bibitem{kiermaier2018order}
Michael Kiermaier, Sascha Kurz, and Alfred Wassermann, \emph{The order of the
  automorphism group of a binary $q$-analog of the {F}ano plane is at most
  two}, Designs, Codes and Cryptography \textbf{86} (2018), no.~2, 239--250.

\bibitem{kleitman1975extremal}
Daniel~J. Kleitman, \emph{On an extremal property of antichains in partial
  orders. the {L}{Y}{M} property and some of its implications and
  applications}, Combinatorics (Hall M. and van Lint~J.H., eds.), NATO Advanced
  Study Institutes Series (Series C -- Mathematical and Physical Sciences),
  vol.~16, Springer, Dordrecht, 1975, pp.~277--290.

\bibitem{koetter2008coding}
Ralf Koetter and Frank~R. Kschischang, \emph{Coding for errors and erasures in
  random network coding}, IEEE Transactions on Information Theory \textbf{54}
  (2008), no.~8, 3579--3591.

\bibitem{kohnert2008construction}
Axel Kohnert and Sascha Kurz, \emph{Construction of large constant dimension
  codes with a prescribed minimum distance}, Mathematical methods in computer
  science, Springer, 2008, pp.~31--42.

\bibitem{segre1964teoria}
Beniamino Segre, \emph{Teoria di galois, fibrazioni proiettive e geometrie non
  desarguesiane}, Annali di Matematica Pura ed Applicata \textbf{64} (1964),
  no.~1, 1--76.

\bibitem{silva2008rank}
Danilo Silva, Frank~R. Kschischang, and Ralf Koetter, \emph{A rank-metric
  approach to error control in random network coding}, IEEE Transactions on
  Information Theory \textbf{54} (2008), no.~9, 3951--3967.

\bibitem{yeung2006network}
Raymond~W. Yeung and Ning Cai, \emph{Network error correction, {I}: Basic
  concepts and upper bounds}, Communications in Information \& Systems
  \textbf{6} (2006), no.~1, 19--35.

\end{thebibliography}

\appendix
\section{Constant-dimension codes in $\F_2^8$ with minimum subspace distance $d=6$ and cardinality $256$}
\label{sec_256}
Let $\mathcal{C}$ be a constant-dimension code in $\F_2^8$ with minimum subspace distance $d=6$ and cardinality 
at least $256$. After the proof of Theorem~\ref{thm_2_8_6_classification} we have already argued that any hyperplane 
$H$ of $\F_2^8$ containing $16$ or $17$ codewords of $\mathcal{C}$ are isomorphic to a type in \cite[Table 2]{heinlein2017classifying}, 
see also Table~\ref{tab:details},  
with index $1$, $7$, or $8$. The possible choices are even more restricted. In \cite{heinlein2017classifying} the authors 
have already determined the $2$ isomorphism types of $31$-configurations for index $1$ and the $240$ isomorphism types of $31$-configurations 
for index $7$ that can yield a constant-dimension code of cardinality at most $256$. Next, we perform some additional considerations 
and computations in order to conclude some results on the structure of $\mathcal{C}$. 

For index $1$ the $31$ solids of the $31$-configurations permit a unique solid $S$ with trivial intersection in both cases. In order to 
obtain some structural results we extend the ILP formulation for $z_8^{\operatorname{BLP}}$. In a first run we prescribe the $31$-configurations 
and $S$ as codewords. Additionally we require that the number of incidences between points of $S$ and codewords is at least $16$. As 
already reported in \cite{heinlein2017classifying}, an upper bound of $252$ could be computationally verified in less than $2$~hours in 
both cases. So, if $S$ is chosen as a codeword, then no other codeword intersects $S$. Removing $S$ yields $255$ codewords disjoint to a 
solid. In a second run we again prescribe the $31$-configurations but forbid the choice of $S$ as a codeword and additionally require 
that at least two codeword non-trivially intersecting $S$ are chosen. This results in $\#\mathcal{C}\le 255$ after roughly $1200$~hours 
of computation time. (We prescribe a codeword non-trivially intersecting $S$ an run the jobs in parallel.) So, $255$ codewords have to be 
disjoint to a solid.

For index $7$ and for each of the $240$ $31$-configurations there exists a unique solid $S$ that is disjoint to $30$ of the $31$ codewords of 
the $31$-configuration and intersects the $31$th codeword in a plane. Prescribing the $31$-configurations and requiring that the number of 
incidences between points of $S$ and codewords is at least $8$ gives an ILP formulations that yields $\#\mathcal{C}\le 242$ after 
a few hours of computation time, see \cite{heinlein2017classifying}. Removing the codeword that intersects $S$ in a plane yields $255$ 
codewords disjoint to a solid.

For index $8$ the ILP $z_{7}^{\text{BLP}}$ was utilized, i.e., after prescribing the $17$ solids corresponding to index $8$ extensions 
by planes in $\F_2^7$, not intersecting themselves and any of the $17$ solids in a line, were searched. The automorphism group of order 
$32$ of the $17$ solids was used to partition the set of $948$ feasible planes into $56$ orbits. In \cite{heinlein2017classifying} it 
was stated that all $56$ cases lead to $z_{7}^{\text{BLP}}\le 256$, which was sufficient for the determination of $A_2(8,6;4)$. 
Here, we tighten the computational results to $z_{7}^{\text{BLP}}\le 255$. To this end we build up 
an ILP formulation using binary variables for the $948$ planes. For each line $L$ at most one plane 
containing $L$ can be chosen. By an explicit inequality we ensure that at least $256-17=239$ planes 
are chosen. For each of the $56$ orbits we try to maximize the number of chosen planes in that orbit 
prescribing one arbitrary plane from that orbit. After ten hours computation time in each of these 
$56$ cases, we obtain upper bounds for the number of chosen planes per orbit. 
Those upper bounds sum to at least $241$ planes in total. So, in a second run, using all computed upper bounds for the 
orbits, we quickly verify the infeasibility of the problem, i.e., no hyperplane with index~$8$ can occur in $\mathcal{C}$. 
Using more but shorter runs iteratively updating the upper bounds for the orbits would reduce the computation time significantly.  

To sum up, a rather small number of configurations can lead to constant-dimension codes of cardinality $256$ only. As argued above, 
at least $255$ codewords have to be disjoint to a given solid. Next we show that those $255$ codewords have to arise from an LMRD 
code by removing one codeword. To that end we show an extendability result for an LMRD code based on arguments using divisible 
linear codes, see \cite{honold2018partial} for an introduction into the application of projective divisible linear codes in the 
characterization of partial spreads. We need a slightly more general approach using multisets of subspaces, see \cite{kiermaier2017improvement}.  
To any multiset $\mathcal{P}$ of points in $\F_q^v$ we associate a characteristic function $\chi_{\mathcal{P}}(P)\in\mathbb{N}$ that 
gives the number of occurrences of $P$ in $\mathcal{P}$.
\begin{definition}
  A multiset  $\mathcal{P}$ of points in $\F_q^v$ characterized by its characteristic function $\chi_{\mathcal{P}}$  
  is called \emph{$q^r$-divisible}, for a non-negative integer $r$, if $\sum_{{P\le\F_q^v}\atop{\dim(P)=1}}\chi_{\mathcal{P}}(P)\equiv 
  \sum_{{P\le H}\atop{\dim(P)=1}}\chi_{\mathcal{P}}(P)\pmod{q^r}$ for any hyperplane $H$ of $\F_q^v$.  
\end{definition}
Note that the  condition for a $q^r$-divisible multiset of points is trivially satisfied for $r=0$. If $r\ge 1$, then the intersection 
of a $q^r$-divisible multiset of points with a hyperplane yields a $q^{r-1}$-divisible multiset of points, see e.g.\ \cite{honold2018partial}. 

If $\mathcal{C}$ is an arbitrary multiset of $k$-subspaces, then replacing every element of $\mathcal{C}$ by its set of $\gauss{k}{1}{q}$  
contained points gives a multiset $\mathcal{P}$ of points. In \cite{kiermaier2017improvement} it was shown that $\mathcal{P}$ is 
$q^{k-1}$-divisible. Moreover, if the corresponding characteristic function satisfies $\chi_{\mathcal{P}}(P)\le \Lambda$ for some 
integer $\Lambda$, then $\Lambda-\chi_{\mathcal{P}}(P)$ is the characteristic function of a $q^{k-1}$-divisible multiset of points. 
The latter construction is also called the $\Lambda$-complement of $\mathcal{P}$.

As a further ingredient we need a few classification results for $q^r$-divisible multisets of points. 

\begin{lemma}
  \label{lemma_simplex_unique}
  For a positive integer $r$, let $\mathcal{P}$ be a $q^r$-divisible multiset of points of cardinality $\gauss{r+1}{1}{q}=\frac{q^{r+1}-1}{q-1}$. 
  Then $\mathcal{P}$ coincides with the set of points of a suitable $(r+1)$-subspace.   
\end{lemma}
\begin{proof}
  Let $v$ denote the dimension of the span of $\mathcal{P}$ and embed $\mathcal{P}$ in $\mathbb{F}_q^v$. Since $\gauss{r+1}{1}{q}-2\cdot q^r$ 
  is negative, every hyperplane of $\mathbb{F}_q^v$ contains exactly $\gauss{r+1}{1}{q}-q^r=\gauss{r}{1}{q}$ points. Double counting 
  the incidences between points and hyperplanes gives $\gauss{r}{1}{q}\cdot\gauss{v}{v-1}{q}=\gauss{r+1}{1}{q}\cdot\gauss{v-1}{v-2}{q}$, 
  so that $v=r+1$. Double counting the incidences between hyperplanes and pairs of points gives that there are no multiple points, so that 
  the statement follows.     
\end{proof}



\begin{lemma}
  \label{lemma_special_1}
  If $\mathcal{P}$ is a $2^3$-divisible multiset of points over $\F_2$ of cardinality $30$ and $\mathcal{L}$ be a set of 
  $70$ lines such that $7\cdot\chi_{\mathcal{P}}(P)=\chi_{\mathcal{P}'}(P)$ for all points $P$, where $\mathcal{P}'$ denotes the 
  multiset of points of the lines from $\mathcal{L}$, then $\mathcal{P}$ equals the union of the set of points of two solids.
\end{lemma}
\begin{proof}
  Let $v$ denote the dimension of the span of $\mathcal{P}$ and embed $\mathcal{P}$ in $\mathbb{F}_2^v$. By $a_i$ we denote the number 
  of hyperplanes of $\F_2^v$ with exactly $i$ points. Due to divisibility only $a_6$, $a_{14}$, and $a_{22}$ can be positive. Any of the 
  $70$ lines intersects each hyperplane in at least one point so that each hyperplane contains at least $70/7=10$ points, which gives  
  $a_6=0$. Counting incidences between points and hyperplanes gives $a_{14}+a_{22}=2^v-1$ and $14a_{14}+22a_{22}=30\cdot \left(2^{v-1}-1\right)$. 
  Setting $y=2^{v-2}$ we obtain $a_{22}=\frac{y}{2}-2$ and $a_{14}=\frac{7y}{2}+1$. If $x$ denotes the number of ordered pairs of elements 
  of $\mathcal{P}$ that correspond to the same point, then double counting the incidences between hyperplanes and pairs of points gives
  $
    14\cdot 13\cdot a_{14}+22\cdot 21\cdot a_{22} =x\cdot (2y-1)+(30\cdot 29-x)\cdot(y-1)
  $,  
  so that $128=y(x+2)$ and
  $(y,x)\in\left\{(4,30), (8,14), (16,6), (32, 2),(64,0)\right\}$. Here $y=4$ corresponds to $v=4$, which allows only $\gauss{4}{2}{2}=35<70$ lines.
  
  Via an ILP formulation we have checked that the cases $y=8$ and $y=16$ are impossible. For $y=32$ and $y=64$ there has to be a solid 
  such that all $15$ contained points have multiplicity at least $1$ in $\mathcal{P}$.
  Since a solid is $2^3$-divisible we can remove the 
  corresponding points of $\mathcal{P}$ and apply Lemma~\ref{lemma_simplex_unique} with $q=2$ and $r=3$ in order to conclude the statement.  
\end{proof}


\begin{proposition}
  Let $\mathcal{C}$ be a set of $254$ or $255$ solids in $\F_2^8$ with minimum subspace distance $d=6$ that intersect a fixed but arbitrary 
  solid $S$ trivially. Then, there exists a solid $U$ in $\F_2^8$ that has subspace distance at least $6$ to the elements of 
  $\mathcal{C}$ and intersects $S$ trivially. 
\end{proposition}
\begin{proof}
  We observe that any point of $\F_2^8$ that is not contained in $S$ is contained in at most $16$ elements from $\mathcal{C}$ 
  and any line from $\F_2^8$ that is not contained in $S$ is contained in at least one element from $\mathcal{C}$. If $\#\mathcal{C}=256$, 
  then all these upper bounds would be attained with equality. 
  
  For $\#\mathcal{C}=255$ a multiset of points and a set of lines of cardinalities 
  $15$ and $35$, respectively, are missing. Now let $\mathcal{P}$ be the multiset of points of the points contained in the elements 
  of $\mathcal{C}$ and the $16$-fold set of points contained in $S$. The $16$-complement of $\mathcal{P}$ is $2^3$-divisible and has 
  cardinality $15$. Let $U$ denote the corresponding solid, see Lemma~\ref{lemma_simplex_unique}. The $35$ lines not covered by the 
  elements from $\mathcal{C}$ that do not lie in $S$ of course all have to be contained in $U$. Thus, $U$ does not share a line with 
  any element from $\mathcal{C}$ and intersects $S$ trivially.  
  
  For $\#\mathcal{C}=254$ a multiset of points and a set of lines of cardinalities $2\cdot 15=30$ and $2\cdot 35=70$, respectively, are 
  missing. Now let $\mathcal{P}$ be the multiset of points of the points contained in the elements of $\mathcal{C}$ and the $16$-fold set 
  of points contained in $S$. The $16$-complement of $\mathcal{P}$, say $\mathcal{P}'$, is $2^3$-divisible and has cardinality $30$. From 
  Lemma~\ref{lemma_special_1} we conclude the existence of two solids $U$ and $U'$ whose union equals $\mathcal{P'}$. By construction 
  both $U$ and $U'$ intersect trivially and have subspace distance at least $6$ to the elements from $\mathcal{C}$. Additionally,  
  $\dim(U\cap U')\le 1$.
\end{proof}
In other words, any set of $254$ or $255$ solids in $\F_2^8$ with minimum subspace distance $6$ that intersect a special solid $S$ 
trivially can be completed to a lifted MRD code.

Given the above extendability result for $\mathcal{G}_{8,4,3}$, we have considered all subsets of $254$ out of the $256$ solids of 
$\mathcal{G}_{8,4,3}$. Up to symmetry there are just $3$ possibilities. Extending these by two further solids in $\F_2^8$ without 
violating the minimum subspace distance $d=6$ gives $4$ non-isomorphic $(8,256,6;4)_2$ constant-dimension codes.  
Three of them can be enlarged in a unique way to a $(8,257,6;4)_2$ constant-dimension code and the fourth allows $451$ such 
extensions. This perfectly matches the classification of the $(8,257,6;4)_2$ codes. Let $S_1$ and $S_2$ denote the solids 
with matrix 
$\left(\begin{smallmatrix}
          0&0&0&0&1&0&0&0\\
          0&0&0&0&0&1&0&0\\
          0&0&0&0&0&0&1&0\\
          0&0&0&0&0&0&0&1
          \end{smallmatrix}\right)$ and 
    $\left(\begin{smallmatrix}
          0&0&0&1&0&0&0&0\\
          0&0&0&0&1&0&0&0\\
          0&0&0&0&0&1&0&0\\
          0&0&0&0&0&0&1&0
          \end{smallmatrix}\right)$, 
respectively. Removing $S_1$ or $S_2$ from $\mathcal{G}_{8,4,3}\cup\{S_1\}$ or $\mathcal{G}_{8,4,3}\cup\{S_2\}$ of course gives the LMRD 
code $\mathcal{G}_{8,4,3}$, which allows $1+\gauss{4}{3}{2}\cdot \frac{255-15}{8}=451$ extensions. In $\mathcal{G}_{8,4,3}\cup\{S_1\}$ every 
other choice of a removal is isomorphic. In $\mathcal{G}_{8,4,3}\cup\{S_2\}$ there are two isomorphism types of removals of a solid $S$ 
separated by $\dim(S\cap S_2)\in\{0,1\}$.     


\begin{proposition}
  \label{prop_256}
  There exist exactly $4$ non-isomorphic $(8,256,6;4)_2$ constant-dimension codes. All of them can be extended to an 
  $(8,257,6;4)_2$ constant-dimension code. 
\end{proposition}

\section{Integer linear programming formulations from \cite{heinlein2017classifying}}

In this section we collect the integer linear programming formulations used in \cite{heinlein2017classifying} to determine 
$A_2(8,6;4)=257$ and to classify the corresponding optimal constant-dimension codes. In the proof of Theorem~\ref{thm_2_8_6_classification} 
we have used a few ILP computations based on Lemma~\ref{lemma_ILP_1} and Lemma~\ref{lemma_ILP_2}. The numerical results, both from 
\cite{heinlein2017classifying} and the newly obtained ones, are summarized in Table~\ref{tab:details}.

For the first ILP formulation we use binary variables $x_U$, where $U$ is an arbitrary $4$-subspace in $\F_2^8$. The interpretation 
is $x_U=1$ if and only if $U$ is a codeword of the desired code $\mathcal{C}$. We prescribe a certain subset $F\subseteq\mathcal{C}$ and assume that the 
number of indicences between codewords and a point or a hyperplane is upper bounded by some integer $f$. As an abbreviation we use the 
notation $U\in \I{\mathcal{S}}{W}$ for all $U\in\mathcal{S}$ with $U\le W$ if $\dim(U)\le \dim(W)$ or $U\ge W$ if $\dim(U)\ge \dim(W)$, i.e., 
those subspaces from $\mathcal{S}$ that are incident with $W$. 

\begin{lemma} \label{lemma_ILP_1}(\cite[Lemma 12]{heinlein2017classifying})\\
Let $\operatorname{Var}_8 = \G{8}{4}{2}$,  $F \subseteq \operatorname{Var}_8$ and $f\in\mathbb{N}$, 
then any $(8,\#\mathcal{C},6;4)_2$ CDC $\mathcal{C}$ with 
$F \subseteq \mathcal{C}$ such that each point and each hyperplane is incident to at 
most $f$ codewords, respectively, satisifies $\#\mathcal{C} \le z_8^{\operatorname{BLP}}(F,f) \le z_8^{\operatorname{LP}}(F,f)$, 
where $z_8^{\operatorname{LP}}$ is the LP relaxation of $z_8^{\operatorname{BLP}}$, and
\begin{align*}
z_8^{\operatorname{BLP}}(F,f) := \max
\sum_{U \in \operatorname{Var}_8} &x_U \\
\st
\sum_{U \in \I{\operatorname{Var}_8}{W}} &x_U \le f			&&\forall W \in \G{8}{w}{2} 	&&\forall w \in \{1,7\} \\
\sum_{U \in \I{\operatorname{Var}_8}{W}} &x_U \le 1			&&\forall W \in \G{8}{w}{2} 	&&\forall w \in \{2,6\} \\
&x_U = 1 													&&\forall U \in F \\
&x_U \in \{0,1\} 											&&\forall U \in \operatorname{Var}_8.
\end{align*}
\end{lemma}

For the second ILP formulation we consider the intersection of a code $\mathcal{C}$ with a hyperplane $H$ as a combinatorial 
relaxation. More precisely, let the hyperplane $H$ be given by the ambient space $\F_2^7$ and $F \subseteq \G{7}{4}{2}$ be the set of 
codewords of $\mathcal{C}$ that are  contained in the hyperplane $H$. All other codewords intersect $H$ in a plane. Those intersections 
are modeled as binary variables $x_U$, where we ensure that $U$ intersects any element from $F$ in at most a point. W.l.o.g.\ we assume 
that the maximum number of incidences between codewords of $\mathcal{C}$ and a point or a hyperplane is attained by $\# F$.

\begin{lemma} \label{lemma_ILP_2}(\cite[Lemma 13]{heinlein2017classifying})\\
For $F \subseteq \G{7}{4}{2}$ let 
$$\operatorname{Var}_7(F):=\left\{ U \in \G{7}{3}{2} \,:\, \dim(U \cap S) \le 1 \,\forall S \in F \right\}$$ and 
$$\omega(F,W) = \max\{ \#\Omega \,:\, \Omega \subseteq \I{\operatorname{Var}_7(F)}{W} \land \dim(U_1 \cap U_2) \le 1\, \forall U_1 \ne U_2 \in \Omega \}.$$ 
If $\#F \in \{16,17\}$, then any $(8,\#\mathcal{C},6;4)_2$ CDC $\mathcal{C}$ with $\#\mathcal{C} \ge 255$ containing the elements 
from $F$ as codewords in the hyperplane $\F_2^7$ such that each point 
and each hyperplane is incident to at most $\#F$ codewords, respectively, satisfies $\#\mathcal{C} \le z_7^{\operatorname{BLP}}(F)$, where
\begin{align*}
z_7^{\operatorname{BLP}}(F) := \max \!\!\!\!\!\!\!
\sum_{U \in \operatorname{Var}_7(F)} &x_U + \#F \\
\st
\sum_{U \in \I{\operatorname{Var}_7(F)}{W}} &x_U \le \#F-\#\I{F}{W}						&&\forall W \in \G{7}{1}{2} \\
\sum_{U \in \I{\operatorname{Var}_7(F)}{W}} &x_U \le 1									&&\forall W \in \G{7}{2}{2} : W \not \le S\,\, \forall S \in F \\  
\sum_{U \in \I{\operatorname{Var}_7(F)}{W}} &x_U \le 1									&&\forall W \in \G{7}{4}{2} \setminus F \\
\sum_{U \in \I{\operatorname{Var}_7(F)}{W}} &x_U \le \min\{\omega(F,W),7\}				&&\forall W \in \G{7}{5}{2} : S \not \le W\,\, \forall S \in F  \\
\sum_{U \in \I{\operatorname{Var}_7(F)}{W}} &x_U \le 2(\#F-\#\I{F}{W})					&&\forall W \in \G{7}{6}{2} \\
\sum_{U \in \operatorname{Var}_7(F)} &x_U + \#F \ge 255 \\
&x_U \in \{0,1\} 																		&&\forall U \in \operatorname{Var}_7(F) \\
\end{align*}
\end{lemma}

\begin{sidewaystable}
\small

 \bigskip
 \bigskip\bigskip\bigskip\bigskip\bigskip\bigskip\bigskip\bigskip\bigskip\bigskip
 \bigskip\bigskip\bigskip\bigskip\bigskip\bigskip\bigskip\bigskip\bigskip\bigskip
 \bigskip\bigskip\bigskip\bigskip\bigskip\bigskip\bigskip\bigskip\bigskip\bigskip
 \bigskip\bigskip\bigskip\bigskip\bigskip\bigskip\bigskip\bigskip\bigskip\bigskip
 \bigskip\bigskip\bigskip\bigskip\bigskip\bigskip\bigskip\bigskip\bigskip\bigskip
 \bigskip\bigskip\bigskip\bigskip\bigskip\bigskip\bigskip\bigskip\bigskip\bigskip

\caption{Details for the ILP computations.}
\label{tab:details}
\begin{tabular}{ll|l|lllll}
Index	&	Type	&	Aut	&	$z_8^{\operatorname{LP}}(.)$	&	$z_7^{\operatorname{BLP}}(.)$	&	Orbits of phase~2	&	$\max z_8^{\operatorname{LP}}(\text{\lq\lq 31\rq\rq})$	&	$\max z_8^{\operatorname{BLP}}(\text{\lq\lq 31\rq\rq})$	\\
\hline
1	&	16	&	960	&	272	&	271.1856	&	$ 16^{2},240^{6},480^{47},960^{242} $	&	263.0287799	&	257	\\
2	&	16	&	384	&	266.26086957	&	267.4646	&	$ 96^{6},192^{91},384^{711} $	&	206.04279728	&		\\
3	&	16	&	4	&	270.83786676	&	265.3281	&	$ 1^{13},2^{29},4^{2638} $	&	257.20717665	&	254	\\
4	&	16	&	48	&	271.43451032	&	262.082	&	$ 4^{3},12^{11},24^{59},48^{1104} $	&	200.5850228	&		\\
5	&	16	&	2	&	263.8132689	&	259.8044	&	$ 1^{5},2^{59966} $	&	206.39304042	&		\\
6	&	16	&	20	&	267.53272206	&	259.394	&	$ 5,10^{9},20^{1843} $	&	199.98690666	&		\\
7	&	17	&	64	&	282.96047431	&	259.1063	&	$ 16^{10},32^{145},64^{6293} $	&	259.45364626	&	257	\\
8	&	17	&	32	&	268.0388109	&	257.2408	& $\le 255$ by a separate argumentation						\\
\end{tabular}

\medskip

\begin{tabular}{ll|l|ll||ll|l|ll}
Index	&	Type	&	Aut	&	$z_8^{\operatorname{LP}}(.)$	&	$z_7^{\operatorname{BLP}}(.)$	& Index	&	Type	&	Aut	&	$z_8^{\operatorname{LP}}(.)$	&	$z_7^{\operatorname{BLP}}(.)$	\\
\hline
9	&	16	&	1	&	263.82742528	&	$\le$ 255   					   &
10	&	16	&	1	&	263.36961743	&	$\le$ 255							\\
11	&	16	&	1	&	264.25957151	&	$\le$ 254							&
12	&	16	&	1	&	263.85869815	&	$\le$ 254							\\
13	&	16	&	2	&	263.07052878	&	$\le$ 254							&
14	&	16	&	12	&	261.91860556	&	$\le$ 254							\\
15	&	16	&	4	&	261.62648174	&	$\le$ 254							&
16	&	16	&	12	&	261.31512837	&	$\le$ 254							\\
17	&	17	&	4	&	261.11518721	&	$\le$ 254							&
18	&	16	&	1	&	260.96388752	&	$\le$ 254							\\
19	&	16	&	1	&	260.82432878	&	$\le$ 254							&
20	&	16	&	2	&	260.65762276	&	$\le$ 254							\\
21	&	16	&	4	&	260.43036283	&	$\le$ 254							&
22	&	16	&	2	&	260.19475349	&	$\le$ 254							\\
23	&	16	&	1	&	260.08583792	&	$\le$ 254							&
24	&	16	&	1	&	260.04857193	&	$\le$ 254							\\
25	&	16	&	1	&	259.75041996	&	$\le$ 254							&
26	&	16	&	2	&	259.55230081	&	$\le$ 254							\\
27	&	16	&	2	&	259.46335297	&	$\le$ 254							&
28	&	16	&	12	&	259.11945025	&	$\le$ 254							\\
29	&	16	&	1	&	258.89395938	&	$\le$ 254							&
30	&	17	&	24	&	258.75142045	&	$\le$ 254							\\
31	&	16	&	8	&	258.35689437	&	$\le$ 254							&
32	&	16	&	1	&	257.81420526	&	$\le$ 254							\\
33	&	16	&	2	&	257.75126819	&	$\le$ 254							&
34	&	16	&	4	&	257.63965018	&	$\le$ 254							\\
35	&	16	&	1	&	257.57663803	&	$\le$ 254							&
36	&	16	&	1	&	257.2820438	&	$\le$ 254							\\
37	&	16	&	4	&	257.01931801	&	$\le$ 254							&
38	&	17	&	128	&	257	&	$\le$ 254							\\
39 &  16 &  12  & 256.83887168	&	$\le$ 254							&
40 &  16 &  12  & 256.31380897	&	$\le$ 254							\\
41 &  16 &  6  & 256.22093781	&	$\le$ 254							&
42 &  16 &  12  & 256.10154389	&	$\le$ 254							\\
43 &  16 &  1  & 255.87957119 &                                &
   &     &     &        & 
\end{tabular}
\end{sidewaystable}

\end{document}